\documentclass[11pt]{amsart}

\usepackage[pdftex]{color,graphicx}
\usepackage[utf8]{inputenc}
\usepackage[english]{babel}
\usepackage{graphics}
\usepackage{enumerate}
\usepackage{array}
\usepackage{verbatim}
\usepackage{hyperref}
\usepackage{caption}
\usepackage{tikz-cd}

\usepackage{multicol}

\usetikzlibrary{graphs,decorations.pathmorphing,decorations.markings}

\usepackage{xcolor}
\definecolor{pink}{HTML}{FF1493}
\usepackage[colorinlistoftodos]{todonotes}

\usepackage{pdflscape}
\usepackage{rotating}
\usepackage{graphicx}

\usepackage{amsmath}
\usepackage{amsfonts}
\usepackage{amssymb}
\usepackage{amsthm}
\usepackage{mathrsfs}
\usepackage{mathtext}
\usepackage{mathtools}
\usepackage[mathscr]{eucal}

\usepackage[all,knot]{xy}

\newcommand{\Z}{\mathbb{Z}}
\newcommand{\N}{\mathbb{N}}

\newcommand{\Uqsl}{U_q(\mathfrak{sl}_2)}
\newcommand{\Uqb}{U_{q^2}(\mathfrak{b})}

\numberwithin{equation}{section}

\theoremstyle{plain}
\newtheorem{theo}[equation]{Theorem}
\newtheorem{prop}[equation]{Proposition}
\newtheorem{coro}[equation]{Corollary}
\newtheorem{lemm}[equation]{Lemma}

\theoremstyle{definition}

\newtheorem{defi}[equation]{Definition}
\newtheorem{exam}[equation]{Example}

\theoremstyle{remark}
\newtheorem{rema}[equation]{Remark}

\title[A dichotomy between twisted bialgebras and twisted {F}robenius algebras]{A dichotomy between twisted tensor products of bialgebras and {F}robenius algebras}

\author[P.\ S.\ Ocal]{Pablo S.\ Ocal}
\address{UCLA Mathematics Department, Los Angeles, CA 90095-1555, USA}
\email{socal@math.ucla.edu}

\author[A.\ Oswald]{Amrei Oswald}
\address{Department of Mathematics, University of Washington, Seattle, WA 98195, USA}
\email{amreio@uw.edu}

\date{November 2022}



\begin{document}


\begin{abstract}
We endow twisted tensor products with a natural notion of counit and comultiplication, and we provide sufficient and necessary conditions making the twisted tensor product a counital coassociative coalgebra. We then characterize when the twisted tensor product of bialgebras is a bialgebra, and when the twisted tensor product of Frobenius algebras is a Frobenius algebra. Our methods are purely diagrammatic, so these results hold for (braided) monoidal categories. As an application, we recover that some quantum complete intersections are Frobenius algebras, and we construct families of noncommutative symmetric Frobenius algebras. Along the way, we also characterize when twisted tensor products of separable algebras are separable, and we prove that twisted tensor products of special Frobenius algebras are special Frobenius.
\end{abstract}

\maketitle

\section{Introduction}\label{sec:introduction}

The study of quantum symmetries is intimately related with the deformation theory of classical objects. It is useful to understand how the properties of the classical objects are inherited, or not, by their deformations. For example, a key idea in the connection between Hopf algebras and solutions of the Yang-Baxter equation is the fact that the representation theory of a quantized universal enveloping algebra of a (complex semisimple) Lie algebra outside roots of unity coincides with the representation theory of the original Lie algebra \cite{MR934283, MR1359532}. This can be exploited further to understand Topological Quantum Field Theories, where both Frobenius algebras and Hopf algebras play significant roles. Namely commutative Frobenius algebras correspond to 2-dimensional TQFT's \cite{MR2695771, MR2037238}, and quantum groups can be used to construct 3-dimensional TQFT's \cite{MR1027945, MR3674995}.

In this paper we consider the deformation of the tensor product of algebras over a field known as twisted tensor product \cite{MR1352565}. These were originally conceived as a noncommutative analogue of the product of topological spaces, but they have an accessible algebraic formulation that encompass vast families of interesting algebras (such as the aforementioned quantized universal enveloping algebras). We seek to describe how the structures of Hopf algebra and Frobenius algebra are inherited by twisted tensor products, if at all. Related ideas have appeared in \cite{MR1755802,MR1926102}.

Our motivating goal is to understand the tensor triangular geometry of Hopf algebras. Given a Hopf algebra $H$, we would like to study the Balmer spectrum \cite{MR2196732} of its associated stable module category in terms of its Hopf subalgebras. In other words, we would like to find appropriate Hopf algebras $A$ and $B$ that are subalgebras of $H$ such that the information provided by $\mathrm{Spc}(\mathrm{stmod}{A})$ and $\mathrm{Spc}(\mathrm{stmod}{B})$ can be used to reconstruct $\mathrm{Spc}(\mathrm{stmod}{H})$. The twisted tensor product serves our purpose as a precise way of encoding how the algebra structure of $H$ is inherited by the algebra structures of $A$ and $B$. To extend this to the coalgebra structure of $H$, we introduce natural candidates for the counit and the comultiplication in terms of the coalgebra structures of $A$ and $B$. Since Frobenius structures also determine a compatibility between the algebra and coalgebra structures of a given vector space, it is natural to ask whether twisted tensor products could inherit a Frobenius algebra structure via the counit and comultiplication we introduce.

Techniques that rely on understanding a twisted tensor product $A\otimes_{\tau} B$ in terms of the algebras $A$ and $B$ have been successfully used on numerous occasions.  Results that can be proved using decompositions as twisted tensor products include the characterization of when the cohomology of quantum complete intersections is finitely generated over the Hochschild cohomology ring \cite{MR2450729}, the fact that Ore extensions preserve Artin-Schelter regularity \cite{MR2452318}, and that crossed products of the quantum plane with the quantized enveloping algebra of $\mathfrak{sl}_2$ admit PBW deformations at primitive third roots of unity \cite{MR3784828}.

Our first main result is that bialgebra structures are not inherited by twisted tensor products with our proposed counit and comultiplication, so Hopf algebra structures will not be inherited either. 

\begin{theo}[see Theorem~\ref{theo:bialgebra-twist-iff-cotwist}]
Let $A$ and $B$ be $k$-bialgebras, let $\tau:B\otimes A\rightarrow A\otimes B$ be a twisting map. Then $A\otimes_{\tau} B$ inherits a $k$-bialgebra structure if and only if $\tau$ is trivial.
\end{theo}

Our second main result is that Frobenius algebra structures are always inherited by twisted tensor products with our proposed counit and comultiplication.

\begin{theo}[see Theorem~\ref{theo:ttp-cttp-always-Frobenius}]
Let $A$ and $B$ be Frobenius algebras over $k$, let $\tau:B\otimes A\rightarrow A\otimes B$ be a twisting map. Then $A\otimes_{\tau} B$ is a Frobenius algebra if and only if $A\otimes_{\tau} B$ is a counital coassociative coalgebra.
\end{theo}

Moreover, we determine when twisted tensor products of separable and special Frobenius algebras are separable or special, respectively. We also recover the fact that certain quantum complete intersections are symmetric Frobenius algebras.

\begin{coro}[see Corollaries~\ref{coro:quantum-complete-Frobenius} and~\ref{coro:p-group-Frobenius}, and Lemma 3.1 in \cite{MR2516162}]
Let $n\in\N$, $\mathbf{m} = (m_1,\dots,m_n)\in\N^n$, $n,m_1,\dots,m_n \geq 2$, and $\mathbf{q} = (q_{ij}) \in M_{n}(k^{\times})$ such that $q_{ii} = 1$ and $q_{ij} q_{ji} = 1$ for all $1\leq i,j\leq n$. If
\begin{enumerate}
\item $q_{ij}$ is a root of unity whose order divides $\gcd(m_i - 1, m_j - 1)$ for all $i,j = 1,\dots,n$, or

\item $k$ is a field of characteristic $p > 0$ and $\mathbf{m} = (p,\dots,p)$,
\end{enumerate}
then the quantum complete intersections $\Lambda_{\mathbf{q},\mathbf{m}}^{n}$ are symmetric Frobenius algebras.
\end{coro}

Finally, we construct noncommutative symmetric Frobenius algebras from finite groups.

\begin{coro}[see Section~\ref{sec:novel-Frobenius}]
Let $G$ and $H$ be finite groups, let $\tau : kH\otimes kG \to kG\otimes kH$ be a non-trivial strongly graded twisting map and denote $\tau(h\otimes g) = \lambda_{h,g} g\otimes h$ for some $\lambda_{h,g} \in k^{\times}$. If $\lambda_{h,g}=\lambda_{s^{-1},g}\lambda_{sh,g}$ and $\lambda_{h,g}=\lambda_{h,r^{-1}}\lambda_{h,rg}$ for all $g,r\in G$ and $h,s\in H$ then $kG\otimes_{\tau} kH$ is a noncommutative symmetric Frobenius algebra.
\end{coro}

We observe that our approach and techniques have potential implications for the study of twisted Segre products \cite{MR4478354} in noncommutative geometry. More precisely, twisted Segre products appear as a subalgebra of twisted tensor products when the twist is strongly graded, in which case diagrammatic proofs automatically carry over. In general, it would be interesting to know when the twisted tensor product of Noetherian algebras is Noetherian.

Throughout this paper, we actively emphasize the use of commutative diagrams. We intentionally give all definitions in terms of maps satisfying certain commutative diagrams, interpret all conditions and properties in terms of commutative diagrams involving maps, and the majority of our proofs rely on the commutativity of the necessary diagrams. Not only is this a remarkably efficient way of working with twisted tensor products, but it also guarantees that all results in Sections~\ref{sec:preliminaries},~\ref{sec:ttp-cttp},~\ref{sec:ttp-bialgebras}, and~\ref{sec:ttp-Frobenius} hold at the level of Hopf objects and Frobenius objects in a (braided) monoidal category $(\mathcal{C},\otimes,1)$. In particular, we extend some of the results in \cite[Section 3.2]{MR2187404}.

\subsection*{Outline}

In Section~\ref{sec:preliminaries}, we establish the definitions, characterizations, and examples of Hopf algebras and Frobenius algebras to be used throughout the paper. In Section~\ref{sec:ttp-cttp}, we recall the definition of twisted tensor products of algebras, introduce natural candidates for counit and comultiplication on twisted tensor products of coalgebras, and establish useful computational tools. In Section~\ref{sec:ttp-bialgebras}, we prove that non-trivial twisted tensor products of bialgebras do not inherit a bialgebra structure. In Section~\ref{sec:ttp-Frobenius}, we characterize when twisted tensor products of Frobenius algebras inherit a Frobenius algebra structure, and we determine the induced pairing and co-pairing. In Section~\ref{sec:novel-Frobenius}, we recover known Frobenius algebra structures on quantum complete intersections, and we provide novel families of noncommutative symmetric Frobenius algebras. Numerous illustrative examples and counterexamples are provided throughout the paper.

\subsection*{Notation}

The following notational conventions will be used in this paper. We will denote by $k$ a fixed field, of arbitrary characteristic unless otherwise stated. Unadorned tensor products are taken to be over $k$, namely $\otimes \coloneqq \otimes_{k}$. The identity morphism will be denoted by $1:V\to V$ for all $k$ vector spaces $V$. The map $\sigma_{ij}:\bigotimes_{l=1}^{m}{V_l} \rightarrow \bigotimes_{l=1}^{i-1}{V_l}\otimes V_j \otimes \bigotimes_{l=i+1}^{j-1}{V_l}\otimes V_i \otimes \bigotimes_{l=j+1}^{m}{V_l}$ denotes the exchange of the $i$-th and $j$-th coordinates. Unless otherwise stated, vector spaces in Section~\ref{sec:novel-Frobenius} will be graded by additive abelian groups, and maps $f:V\to W$ between graded $k$ vector spaces are assumed to be homogeneous of some degree $d\in\Z$. Namely, if $f(v) = \sum_{i\in I}{w_i}$ for some indexing set $I$, $v\in V$, and $w_i\in W$ for all $i\in I$, then $|v| + d = |w_i|$ for all $i\in I$. The Kronecker delta will be denoted by $\delta_{i,j}$. We will use Sweedler's notation for the comultiplication. Finally, exclusively within diagrams, we omit tensor products to economize space and composition is denoted by $\circ$. Outside commutative diagrams, composition is denoted by concatenation. Namely given $k$ vector spaces $X,Y,U,V$ and $k$-linear morphisms $f:X\to Y$, $g:U\to V$, $h:Y\to U$, then within commutative diagrams $UV \coloneqq U\otimes V$ and $fg \coloneqq f\otimes g$. We use \fbox{?} to indicate a diagram that has not yet been shown to commute.

\section{Preliminaries}\label{sec:preliminaries}

In this section, we introduce the definitions and background necessary to understand the majority of the statements and results. These are presented diagrammatically to facilitate the proofs in the following sections. We refer the reader to \cite{MR1653294, MR2037238, MR2894855} for the details.

\begin{defi}
An associative, unital $k$-\emph{algebra} is a triple $(A,\nabla,\eta)$ where $A$ is a $k$ vector space, and $\nabla: A\otimes A \to A$ and $\eta: k \to A$ are $k$-linear maps making the following diagrams commute.

\begin{equation*}
\begin{tikzcd}
A\otimes k \arrow{r}{1\otimes\eta} \arrow[swap]{dr}{\cong} & A\otimes A \arrow{d}{\nabla}\\
 & A
\end{tikzcd} \quad
\begin{tikzcd}
k\otimes A \arrow{r}{\eta\otimes 1} \arrow[swap]{dr}{\cong} & A\otimes A \arrow{d}{\nabla}\\
 & A
\end{tikzcd} \quad
\begin{tikzcd}
A\otimes A\otimes A \arrow{r}{\nabla\otimes 1}  \arrow[swap]{d}{1\otimes\nabla} & A\otimes A \arrow{d}{\nabla}\\
A\otimes A \arrow{r}{\nabla} & A
\end{tikzcd}
\end{equation*}
\end{defi}

\begin{defi}
A $k$-algebra $(A,\nabla,\eta)$ is \emph{separable} when the $A$ bimodule morphism $\nabla : A\otimes A\to A$ has a right inverse. That is, there is an $k$-linear map $\Gamma : A \to A\otimes A$ making the following diagrams commute.
\begin{center}
\begin{tikzcd}
A \arrow{r}{\Gamma} \arrow[swap]{rd}{1} & A\otimes A \arrow{d}{\nabla} \\
& A
\end{tikzcd}\quad \begin{tikzcd}
A \otimes A \otimes A \arrow[swap]{d}{1 \otimes \Gamma \otimes 1} \arrow{r}{1 \otimes \nabla} & A\otimes A \arrow{r}{\nabla} & A \arrow{d}{\Gamma}\\
A \otimes A\otimes A\otimes A \arrow{rr}{\nabla \otimes \nabla} & & A\otimes A
\end{tikzcd}
\end{center}
\end{defi}

The right diagram above states that $\Gamma : A \to A\otimes A$ is an $A$ bimodule morphism.

\begin{defi}
A coassociative, counital $k$-\emph{coalgebra} is a triple $(C,\Delta,\epsilon)$ where $C$ is a $k$ vector space, and $\Delta: C \to C\otimes C$ and $\epsilon: C \to k$ are $k$-linear maps making the following diagrams commute.
\begin{equation*}
\begin{tikzcd}
A \arrow{r}{\Delta} \arrow[swap]{dr}{\cong}& A\otimes A\arrow{d}{1\otimes\epsilon}\\
 & A\otimes k
\end{tikzcd} \quad
\begin{tikzcd}
A \arrow{r}{\Delta} \arrow[swap]{dr}{\cong}& A\otimes A\arrow{d}{\epsilon\otimes 1}\\
 & k\otimes A
\end{tikzcd} \quad
\begin{tikzcd}
A \arrow{r}{\Delta} \arrow[swap]{d}{\Delta} & A\otimes A \arrow{d}{1\otimes\Delta}\\
A\otimes A \arrow{r}{\Delta\otimes 1} & A\otimes A\otimes A
\end{tikzcd}
\end{equation*}
\end{defi}

\begin{defi}\label{defi:bialgebra}
A $k$-\emph{bialgebra} is a tuple $(A,\nabla,\eta,\Delta,\epsilon)$ where $(A,\nabla,\eta)$ is a $k$-algebra, $(A,\Delta,\epsilon)$ is a $k$-coalgebra, and the following diagrams commute.
\begin{equation*}
\begin{tikzcd}
A\otimes A \arrow{r}{\nabla} \arrow[swap]{d}{\epsilon \otimes \epsilon} & A \arrow{d}{\epsilon}\\
k\otimes k \arrow{r}{\cong} & k
\end{tikzcd}\quad
\begin{tikzcd}
k \arrow{r}{\cong} \arrow[swap]{d}{\eta} & k\otimes k \arrow{d}{\eta\otimes \eta}\\
A \arrow{r}{\Delta} & A\otimes A
\end{tikzcd}
\end{equation*}
\begin{equation*}
\begin{tikzcd}
k \arrow{r}{\eta} \arrow[swap]{dr}{1} & A \arrow{d}{\epsilon}\\
 & k
\end{tikzcd}\quad
\begin{tikzcd}
A\otimes A \arrow{r}{\nabla} \arrow[swap]{d}{\Delta\otimes\Delta} & A \arrow{r}{\Delta} & A\otimes A\\
A\otimes A\otimes A\otimes A \arrow{rr}{\sigma_{23}} & & A\otimes A\otimes A\otimes A \arrow[swap]{u}{\nabla\otimes \nabla}
\end{tikzcd}
\end{equation*}
\end{defi}

The examples in this paper will mostly use the following two bialgebra structures.

\begin{exam}\label{exam:bialgebra}
Let $G$ be a finite group, the group algebra $kG$ with the usual unit and multiplication, and counit $\epsilon:kG\to k$ and comultiplication $\Delta:kG\to kG\otimes kG$ given by extending $\epsilon(g) = 1$ and $\Delta(g) = g\otimes g$ for all $g\in G$ is a bialgebra. The polynomial ring in one variable $k[x]$ with the usual unit and multiplication, and counit $\epsilon:k[x]\to k$ and comultiplication $\Delta:k[x]\to k[x]\otimes k[x]$ given by extending $\epsilon(1) = 1$, $\epsilon(x) = 0$, and $\Delta(x) = 1\otimes x + x\otimes 1$ is a bialgebra.
\end{exam}

\begin{defi}\label{defi:frob-alg-comult}
A \emph{Frobenius algebra} over $k$ is a tuple $(A,\nabla,\eta,\Delta,\epsilon)$ where $(A,\nabla,\eta)$ is a $k$-algebra, $(A,\Delta,\epsilon)$ is a $k$-coalgebra, and the following diagrams commute.
\begin{equation}\label{eq:frob-condition}
\begin{tikzcd}
A\otimes A \arrow{r}{\nabla} \arrow[swap]{d}{\Delta\otimes 1} & A \arrow{d}{\Delta}\\
A\otimes A\otimes A \arrow{r}{1\otimes\nabla} & A\otimes A
\end{tikzcd}\quad
\begin{tikzcd}
A\otimes A \arrow{r}{\nabla} \arrow[swap]{d}{1\otimes\Delta} & A \arrow{d}{\Delta}\\
A\otimes A\otimes A \arrow{r}{\nabla\otimes 1} & A\otimes A
\end{tikzcd}
\end{equation} 
\end{defi}

The examples in this paper will mostly use the following two Frobenius algebra structures.

\begin{exam}\label{exam:Frobenius}
Let $G$ be a finite group, the group algebra $kG$ with the usual unit and multiplication, and counit $\epsilon:kG\to k$ and comultiplication $\Delta:kG\to kG\otimes kG$ given by extending
\begin{equation*}
\epsilon(g) = \delta_{g,1},\quad \Delta(g) = \sum_{r\in G}{rg\otimes r^{-1}}
\end{equation*}
is a Frobenius algebra. Let $n\in\N$, the truncated polynomial ring in one variable $k[x]/(x^n)$ with the usual unit and multiplication, and counit $\epsilon:k[x]/(x^n)\to k$ and comultiplication $\Delta:k[x]/(x^n)\to k[x]/(x^n)\otimes k[x]/(x^n)$ given by extending
\begin{equation*}
\epsilon(x^i) = \delta_{i,n-1},\quad \Delta(p(x)) = \sum_{j=0}^{n-1}{x^j p(x) \otimes x^{n-1-j}}
\end{equation*}
for all $i=1,\dots,n-1$ and all $p(x)\in k[x]/(x^n)$ is a Frobenius algebra.
\end{exam}

The characterization of Frobenius algebras in Definition~\ref{defi:frob-alg-comult} favors our approach seeking to emphasize the use of commutative diagrams. It is well known that there are many equivalent definitions, see for example \cite[Chapter 6]{MR1653294} for an extensive review. Another predominant viewpoint in the literature of noncommutative algebra follows, see for example \cite[Section 2.1]{MR2695771}.

\begin{defi}
Let $(A,\nabla,\eta)$ be an algebra. A \emph{pairing} is a $k$-linear map $\beta: A\otimes A \to k$. A \emph{co-pairing} is a $k$-linear map $\alpha: k \to A\otimes A$. A pairing $\beta$ is said to be \emph{non-degenerate} when there exists a co-pairing $\alpha$ such that the following diagram commutes.
\begin{equation}\label{eq:non-degen}
\begin{tikzcd}
k\otimes A \arrow{r}{\alpha\otimes 1} \arrow[swap]{d}{\cong} & A\otimes A\otimes A \arrow{r}{1\otimes\beta} & A\otimes k \arrow{d}{\cong}\\
A \arrow{rr}{1} & & A
\end{tikzcd}
\end{equation}
A pairing $\beta$ is said to be \emph{associative} when the following diagram commutes.
\begin{equation}\label{eq:assoc}
\begin{tikzcd}
A\otimes A \otimes A \arrow{r}{1\otimes \nabla} \arrow[swap]{d}{\nabla\otimes 1} & A \otimes A \arrow{d}{\beta}\\
A\otimes A\arrow{r}{\beta} & k
\end{tikzcd}
\end{equation}
\end{defi}

\begin{prop}\label{prop:frob-alg-pairing}
A $k$-algebra $(A,\nabla,\eta)$ is a Frobenius algebra if and only if there exists an associative non-degenerate pairing $\beta: A\otimes A \to k$.
\end{prop}

\begin{rema}\label{rema:frob-alg-p-c-ring}
A Frobenius algebra $(A,\nabla,\eta,\Delta,\epsilon)$ has associative non-degenerate pairing $\beta$ and co-pairing $\alpha$ given by
\begin{center}
\begin{tikzcd}
\beta: A\otimes A  \arrow{r}{\nabla} & A \arrow{r}{\epsilon} & k,
\end{tikzcd}
\end{center}
\begin{center}
\begin{tikzcd}
\alpha: k \arrow{r}{\eta} & A \arrow{r}{\Delta} & A\otimes A.
\end{tikzcd}
\end{center}
A Frobenius algebra $(A,\nabla,\eta)$ with an associative non-degenerate pairing $\beta: A\otimes A \to k$ and co-pairing $\alpha:k \to A\otimes A$ has comultiplication(s) $\Delta$ and counit(s) $\epsilon$ given by the following commutative diagrams.
\begin{center}
\begin{tikzcd}
A\otimes k \arrow{r}{1\otimes\alpha} & A \otimes A\otimes A \arrow{d}{\nabla\otimes 1}\\
A \arrow{u}{\cong} \arrow[swap]{d}{\cong} \arrow{r}{\Delta} & A\otimes A\\
k\otimes A \arrow{r}{\alpha\otimes 1} & A \otimes A\otimes A \arrow[swap]{u}{1\otimes\nabla}
\end{tikzcd}\quad
\begin{tikzcd}
A\otimes k \arrow{r}{1\otimes \eta} & A\otimes A \arrow{d}{\beta}\\
A \arrow{u}{\cong} \arrow[swap]{d}{\cong} \arrow{r}{\epsilon} & k\\
k\otimes A \arrow{r}{\eta\otimes 1} & A\otimes A \arrow[swap]{u}{\beta}
\end{tikzcd}
\end{center}
\end{rema}

\begin{defi}
A Frobenius algebra $A$ with associative non-degenerate pairing $\beta:A\otimes A \to k$ is \emph{symmetric} when the following diagram commutes.
\begin{center}
\begin{tikzcd}
A\otimes A \arrow{r}{\sigma_{12}} \arrow[swap]{rd}{\beta} & A\otimes A \arrow{d}{\beta}\\
& k
\end{tikzcd}
\end{center}
\end{defi}

\begin{defi}
A Frobenius algebra $(A,\nabla,\eta,\Delta,\epsilon)$ is \emph{special} when the following diagram commutes.
\begin{center}
\begin{tikzcd}
A \arrow{r}{\Delta} \arrow[swap]{rd}{1} & A\otimes A \arrow{d}{\nabla}\\
& A
\end{tikzcd}
\end{center}
\end{defi}

\begin{prop}\label{prop:frob-special-separable-comultiplication}
Let $(A,\nabla,\eta,\Delta,\epsilon)$ be a Frobenius algebra. Then $A$ is special if and only if $(A,\nabla,\eta)$ is a separable algebra where $\Delta$ is a right inverse of $\nabla$.
\end{prop}
\begin{proof}
Clearly a special Frobenius algebra is separable as claimed. For the converse, it suffices to show that the comultiplication is a morphism of $A$ bimodules. The commutativity of the following diagram guarantees that.
\begin{equation}
\begin{tikzcd}
AAA \arrow{r}{\nabla 1} \arrow[swap]{dd}{1 \Delta 1} & AA \arrow{r}{\nabla} \arrow{d}{\Delta 1} & A \arrow{dd}{\Delta}\\
 & AAA \arrow{dr}{1 \nabla} & \\
AAAA \arrow{rr}{\nabla \nabla} \arrow{ur}{\nabla 1} & & AA
\end{tikzcd}
\end{equation}
\end{proof}

\section{Twisted tensor products as algebras and coalgebras}\label{sec:ttp-cttp}

In this section, we recall the definition of twisted tensor products in the sense of \v{C}ap, Schichl, and Van\v{z}ura \cite{MR1352565}, and we endow them with natural notions of counit and comultiplication. Unless otherwise stated, the (co)algebra structure(s) on twisted tensor products will be the one(s) inherited from the twisting map for the remainder of this paper.

\begin{defi}\label{defi:twisting-map}
Let $(A,\nabla_A,\eta_A)$ and $(B,\nabla_B,\eta_B)$ be unital associative $k$-algebras. Let $\tau:B\otimes A\rightarrow A\otimes B$ be a bijective $k$-linear map such that the following diagrams commute.
\begin{equation}\label{diag:tau-units}
\begin{tikzcd}
B\otimes k \arrow[swap]{d}{1\otimes \eta_A} \arrow{r}{\cong} & k\otimes B \arrow{d}{\eta_A\otimes 1} \\
B\otimes A \arrow{r}{\tau} & A\otimes B
\end{tikzcd} \quad \begin{tikzcd}
k\otimes A \arrow[swap]{d}{\eta_B\otimes 1} \arrow{r}{\cong} & A\otimes k \arrow{d}{1\otimes \eta_B} \\
B\otimes A \arrow{r}{\tau} & A\otimes B
\end{tikzcd}
\end{equation}
\begin{equation}\label{diag:tau-multiplications}
\begin{tikzcd}
B\otimes B\otimes A\otimes A  \arrow[swap]{d}{1\otimes \tau \otimes 1} \arrow{r}{\nabla_B \otimes \nabla_A} & B\otimes A \arrow{r}{\tau} & A\otimes B \\
B\otimes A \otimes B\otimes A \arrow{r}{\tau\otimes \tau} & A\otimes B\otimes A\otimes B \arrow{r}{1\otimes \tau\otimes 1} & A\otimes A\otimes B\otimes B \arrow[swap]{u}{\nabla_A\otimes \nabla_B}
\end{tikzcd}
\end{equation}
We say that $\tau$ is a \emph{twisting map}. The trivial twisting map is $\sigma_{12}:B\otimes A\rightarrow A\otimes B$. When $A$ and $B$ are $k$-algebras graded by commutative groups $F$ and $G$ respectively, a twisting map $\tau:B\otimes A\to A\otimes B$ is said to be \emph{strongly graded} provided $\tau(B_j\otimes A_i) \subseteq A_i\otimes B_j$ for all $i\in F$ and $j\in G$.
\end{defi}

These diagrams encode the compatibility conditions of a twisting map $\tau$ with the unital associative structures of $A$ and $B$. Diagram \eqref{diag:tau-units} can be read as stating that $\tau$ preserves the units $\eta_A$ and $\eta_B$ of $A$ and $B$ respectively. Diagram \eqref{diag:tau-multiplications} is equivalent to $\tau$ preserving the multiplications $\nabla_A$ and $\nabla_B$ of $A$ and $B$ respectively.

\begin{lemm}\label{lemm:twist-multiplication-compatibility}
Let $(A,\nabla_A,\eta_A)$ and $(B,\nabla_B,\eta_B)$ be unital associative $k$-algebras and let $\tau:B\otimes A\rightarrow A\otimes B$ be a twisting map. Then the following diagrams commute.
\begin{equation}\label{diag:tau-commutes-mA}
\begin{tikzcd}
B\otimes A\otimes A  \arrow{r}{\tau\otimes 1} \arrow[swap]{d}{1\otimes \nabla_A} & A\otimes B\otimes A \arrow{r}{1\otimes \tau} & A\otimes A\otimes B \arrow{d}{\nabla_A\otimes 1} \\
B\otimes A \arrow{rr}{\tau} & & A\otimes B
\end{tikzcd}
\end{equation}
\begin{equation}\label{diag:tau-commutes-mB}
\begin{tikzcd}
B\otimes B\otimes A  \arrow{r}{1\otimes \tau} \arrow[swap]{d}{\nabla_B\otimes 1} & B\otimes A\otimes B \arrow{r}{\tau\otimes 1} & A\otimes B\otimes B \arrow{d}{1\otimes\nabla_B}\\
B\otimes A \arrow{rr}{\tau} & & A\otimes B
\end{tikzcd}
\end{equation}
\end{lemm}
\begin{proof}
The following commutative diagram shows the commutativity of \eqref{diag:tau-commutes-mA}.
\begin{center}
\begin{tikzcd}
& BAA \arrow{rrr}{\tau 1} \arrow[swap, bend right=60]{ddd}{1\nabla_A} \arrow{d}{\cong} \arrow{dr}{\cong} & \phantom{x}  & \phantom{x}  & ABA \arrow{r}{1\tau} \arrow{d}{\cong} & AAB \arrow{d}{\cong} \arrow[bend left=60]{ddd}{\nabla_A 1} & \\
& BkAA \arrow{d}{1\eta_B 11} \arrow{r}{\cong} & BAkA \arrow{r}{\tau 11} \arrow{d}{11\eta_B 1} & ABkA \arrow{d}{11\eta_B 1} \arrow{r}{\cong} \arrow{ur}{\cong} & ABAk \arrow{r}{1\tau 1} \arrow{d}{111\eta_B} & AABk \arrow[swap]{d}{111\eta_B} & \\
\phantom{x} & BBAA \arrow{r}{1\tau 1} \arrow{d}{\nabla_B \nabla_A} & BABA \arrow{r}{\tau 11} & ABBA \arrow{r}{11\tau} & ABAB \arrow{r}{1\tau 1} & AABB \arrow[swap]{d}{\nabla_A \nabla_B} & \phantom{x}\\
& BA \arrow{rrrr}{\tau} & & & & AB &
\end{tikzcd}
\end{center}
The commutativity of \eqref{diag:tau-commutes-mB} follows analogously.
\end{proof}

In fact, the above shows that a bijective $k$-linear map $\tau:B\otimes A\rightarrow A\otimes B$ is a twisting map if and only if diagrams \eqref{diag:tau-units}, \eqref{diag:tau-commutes-mA}, and \eqref{diag:tau-commutes-mB} commute.

\begin{defi}\label{defi:twisted-tensor-product}
Let $(A,\nabla_A,\eta_A)$ and $(B,\nabla_B,\eta_B)$ be unital associative $k$-algebras and let $\tau:B\otimes A\rightarrow A\otimes B$ be a twisting map. The \emph{twisted tensor product algebra} $A\otimes_{\tau} B$ is the $k$ vector space $A\otimes B$ with the following multiplication and unit.
\begin{center}
\begin{tikzcd}
\nabla_{A\otimes_{\tau} B}: (A\otimes B)\otimes (A\otimes B)  \arrow{r}{1\otimes \tau \otimes 1} & A\otimes A\otimes B\otimes B \arrow{r}{\nabla_A\otimes\nabla_B} & A\otimes B
\end{tikzcd}
\end{center}
\begin{center}
\begin{tikzcd}
\eta_{A\otimes_{\tau} B}: k \arrow{r}{\cong} & k\otimes k  \arrow{r}{\eta_A\otimes \eta_B} & A\otimes B
\end{tikzcd}
\end{center}
\end{defi}

\begin{prop} \label{prop:AtB-unital-associative}
Let $(A,\nabla_A,\eta_A)$ and $(B,\nabla_B,\eta_B)$ be unital associative $k$-algebras and let $\tau:B\otimes A\rightarrow A\otimes B$ be a twisting map. Then $(A\otimes_{\tau} B,\nabla_{A\otimes_{\tau} B},\eta_{A\otimes_{\tau} B})$ is a unital associative $k$-algebra.
\end{prop}
\begin{proof}
Left unitality follows from the following commutative diagram.
\begin{center}
\begin{tikzcd}
kAB \arrow[swap, bend right=60]{ddd}{\eta_{A\otimes_{\tau} B} 11} \arrow{rrr}{\cong} \arrow{d}{\cong} & & & AB \arrow{d}{\cong} \\
kkAB \arrow{r}{\cong} \arrow{d}{\eta_A 111} & kAkB \arrow{r}{\cong} \arrow{d}{\eta_A 111} & AkB \arrow{r}{\cong} \arrow{d}{\cong} & AB \arrow{d}{\cong} \\
AkAB \arrow{d}{1\eta_B 11} \arrow{r}{\cong} & AAkB \arrow{d}{11\eta_B 1} \arrow{r}{\nabla_A 1} & AkB \arrow{d}{1\eta_B 1} \arrow{r}{\cong} & AB \arrow{d}{\cong} \\
ABAB \arrow{r}{1\tau 1} & AABB \arrow{r}{\nabla_A 11} & ABB \arrow{r}{1 \nabla_B} & AB
\end{tikzcd}
\end{center}
Right unitality follows analogously. Associativity follows from the following commutative diagram.
\begin{center}
\begin{tikzcd}
ABABAB \arrow[swap, bend right=60]{dd}{\nabla_{A\otimes_{\tau} B} 1} \arrow[swap, bend left=20]{rrr}{1 \nabla_{A\otimes_{\tau} B}} \arrow{r}{111\tau 1} \arrow{d}{1\tau 111} & ABAABB \arrow{rr}{11\nabla_A \nabla_B} \arrow{d}{1\tau 111} & & ABAB \arrow{d}{1\tau 1} \arrow[bend left=60]{dd}{\nabla_{A\otimes_{\tau} B}} \\
AABBAB \arrow{r}{111 \tau 1} \arrow{d}{\nabla_A \nabla_B11} & AABABB \arrow{r}{11\tau 11} & AAABBB \arrow{r}{\nabla_A 1 \nabla_B 1} \arrow{d}{\nabla_A 1 \nabla_B 1} & AABB \arrow[swap]{d}{\nabla_A \nabla_B} \\
ABAB \arrow[bend right=20]{rrr}{\nabla_{A\otimes_{\tau} B}} \arrow{rr}{1\tau 1} & & AABB \arrow{r}{\nabla_A \nabla_B} & AB\\
\end{tikzcd}
\end{center}
The original proof is \cite[Proposition/Definition 2.3]{MR1352565}.
\end{proof}

As before, the above shows that given a bijective $k$-linear map $\tau:B\otimes A\rightarrow A\otimes B$ then $(A\otimes_{\tau} B,\nabla_{A\otimes_{\tau} B},\eta_{A\otimes_{\tau} B})$ is a unital associative algebra if and only if \eqref{diag:tau-units} and \eqref{diag:tau-multiplications} commute, equivalently if and only if $\tau$ is a twisting map.

\begin{exam}[Twisting by a bicharacter] \label{exam:twisting-bicharacter} \cite[Definition 2.2]{MR2450729}
Let $A$ and $B$ be $k$-algebras graded by abelian groups $F$ and $G$ respectively, let $t:F\otimes_{\mathbb{Z}} G \rightarrow k^{\times}$ be a homomorphism of abelian groups and denote $t(f\otimes_{\mathbb{Z}} g) = t^{\langle f | g\rangle}$ for all $f\in F$ and $g\in G$. Then $\tau:B\otimes A\rightarrow A\otimes B$ given by linearly extending $\tau(b\otimes a) = t^{\langle |a| | |b|\rangle} a\otimes b$ for all homogeneous $a\in A$ and $b\in B$ is a twisting map. We denote $A\otimes^{t} B \coloneqq A\otimes_{\tau} B$.
\end{exam}

So far, we have only required that $A$ and $B$ are unital associative algebras. When they are also counital coassociative coalgebras, a twisting map $\tau$ induces natural candidates for counit and comultiplication in $A\otimes B$.

\begin{defi}\label{rema:coalg-defi}
Let $(A,\Delta_A,\epsilon_A)$ and $(B,\Delta_B,\epsilon_B)$ be counital coassociative $k$-coalgebras and let $\tau:B\otimes A\rightarrow A\otimes B$ be a $k$-linear map. We define $\Delta_{A\otimes_{\tau} B}$ and $\epsilon_{A\otimes_{\tau} B}$ as
\begin{center}
\begin{tikzcd}
\Delta_{A\otimes_{\tau} B}: A\otimes B  \arrow{r}{\Delta_A\otimes \Delta_B} & A\otimes A\otimes B\otimes B \arrow{r}{1\otimes \tau^{-1}\otimes 1} & (A\otimes B) \otimes (A\otimes B),
\end{tikzcd}
\end{center}
\begin{center}
\begin{tikzcd}
\epsilon_{A\otimes_{\tau} B}: A\otimes B \arrow{r}{\epsilon_A\otimes \epsilon_B} & k\otimes k  \arrow{r}{\cong} & k.
\end{tikzcd}
\end{center}
\end{defi}

\begin{prop}\label{prop:AtB-counital-coassociative}
Let $(A,\Delta_A,\epsilon_A)$ and $(B,\Delta_B,\epsilon_B)$ be counital coassociative $k$-coalgebras and let $\tau:B\otimes A\rightarrow A\otimes B$ be a $k$-linear map. Then $(A\otimes_{\tau} B,\Delta_{A\otimes_{\tau} B},\epsilon_{A\otimes_{\tau} B})$ is a counital coassociative coalgebra if and only if the following diagrams commute.
\begin{equation}\label{diag:tau-counits}
\begin{tikzcd}
B\otimes A \arrow{r}{\tau} \arrow[swap]{d}{1\otimes \epsilon_A} & A\otimes B \arrow{d}{\epsilon_A\otimes 1} \\
B\otimes k \arrow{r}{\cong} & k\otimes B
\end{tikzcd} \quad \begin{tikzcd}
B\otimes A \arrow{r}{\tau} \arrow[swap]{d}{\epsilon_B\otimes 1} & A\otimes B \arrow{d}{1\otimes \epsilon_B} \\
k\otimes A \arrow{r}{\cong} & A\otimes k
\end{tikzcd}
\end{equation}
\begin{equation}\label{diag:tau-comultiplications}
\begin{tikzcd}
B\otimes A \arrow{r}{\tau} \arrow[swap]{d}{\Delta_B\otimes \Delta_A} & A\otimes B \arrow{r}{\Delta_A\otimes \Delta_B} & A\otimes A\otimes B\otimes B\\
B\otimes B \otimes A\otimes A \arrow{r}{1\otimes\tau \otimes 1} & B\otimes A\otimes B\otimes A \arrow{r}{\tau \otimes \tau} & A\otimes B\otimes A\otimes B \arrow[swap]{u}{1\otimes \tau \otimes 1}
\end{tikzcd}
\end{equation}
\end{prop}
\begin{proof}
It follows from reversing the arrows in the proof of Proposition~\ref{prop:AtB-unital-associative}.
\end{proof}

Reversing the arrows in the proof of Lemma~\ref{lemm:twist-multiplication-compatibility} gives that diagram \eqref{diag:tau-comultiplications} commutes if and only if diagrams \eqref{diag:tau-commutes-cmA} and \eqref{diag:tau-commutes-cmB} commute.
\begin{equation}\label{diag:tau-commutes-cmA}
\begin{tikzcd}
B\otimes A \arrow{rr}{\tau} \arrow[swap]{d}{1\otimes \Delta_A} & & A\otimes B \arrow{d}{\Delta_A\otimes 1} \\
B\otimes A\otimes A  \arrow{r}{\tau \otimes 1} & A\otimes B\otimes A \arrow{r}{1 \otimes \tau} & A\otimes A\otimes B
\end{tikzcd}
\end{equation}
\begin{equation}\label{diag:tau-commutes-cmB}
\begin{tikzcd}
B\otimes A \arrow{rr}{\tau} \arrow[swap]{d}{\Delta_B \otimes 1} & & A\otimes B \arrow{d}{1\otimes \Delta_B} \\
B\otimes B\otimes A  \arrow{r}{1\otimes\tau} & B\otimes A\otimes B \arrow{r}{\tau\otimes 1} & A\otimes B\otimes B
\end{tikzcd}
\end{equation}

For computational and technical purposes, it is useful to note that consecutive applications of twisting maps yield the same result. Namely setting
\begin{center}
\begin{tikzcd}
\tau_{2,A}: B\otimes B\otimes A  \arrow{r}{1\otimes \tau} & B\otimes A\otimes B \arrow{r}{\tau\otimes 1} & A\otimes B \otimes B
\end{tikzcd}
\end{center}
\begin{center}
\begin{tikzcd}
\tau_{B,2}: B\otimes A\otimes A  \arrow{r}{\tau\otimes 1} & A\otimes B\otimes A \arrow{r}{1\otimes \tau} & A\otimes A \otimes B
\end{tikzcd}
\end{center}
and recursively considering
\begin{center}
\begin{tikzcd}
\tau_{i,A}: B^{\otimes i} \otimes A  \arrow{rr}{1\otimes \tau_{i-1,A}} & & B\otimes A\otimes B^{\otimes (i-1)} \arrow{rr}{\tau\otimes 1^{\otimes (i-1)}} & & A\otimes B^{\otimes i}
\end{tikzcd}
\end{center}
\begin{center}
\begin{tikzcd}
\tau_{B,j}: B \otimes A^{\otimes j}  \arrow{rr}{\tau_{B,j-1}\otimes 1} & & A^{\otimes (j-1)}\otimes B\otimes A \arrow{rr}{1^{\otimes (j-1)}\otimes \tau} & & A^{\otimes j}\otimes B
\end{tikzcd}
\end{center}
for all $i,j\in \N$, we have the following commutative diagrams.
\begin{center}
\begin{tikzcd}
B^{\otimes (i-1)} \otimes A^{\otimes j} \otimes B \arrow{rr}{1^{\otimes (i-2)}\otimes \tau_{B,j}\otimes 1} & & \cdots \arrow{rr}{1\otimes \tau_{B,j}\otimes 1^{\otimes (i-2)}} & & B \otimes A^{\otimes j} \otimes B^{\otimes (i-1)} \arrow{d}{\tau_{B,j}\otimes 1^{\otimes (i-1)}}\\
B^{\otimes i} \otimes A^{\otimes j}  \arrow{u}{1^{\otimes (i-1)}\otimes \tau_{B,j}} \arrow[swap]{d}{\tau_{i,A}\otimes 1^{\otimes (j-1)}} \arrow[dashed]{rrrr}{\tau_{i,j}} & & & & A^{\otimes j}\otimes B^{\otimes i}\\
A\otimes B^{\otimes i} \otimes A^{\otimes (j-1)} \arrow{rr}{1\otimes \tau_{i,A}\otimes 1^{\otimes (j-2)}} & & \cdots \arrow{rr}{1^{\otimes (j-2)}\otimes \tau_{i,A}\otimes 1} & & A^{\otimes (j-1)} \otimes B^{\otimes i} \otimes A \arrow[swap]{u}{1^{\otimes (j-1)}\otimes \tau_{i,A}}
\end{tikzcd}
\end{center}
In particular, any such consecutive application of twisting maps deserves to be denoted $\tau_{i,j}:B^{\otimes i}\otimes A^{\otimes j} \rightarrow A^{\otimes j}\otimes B^{\otimes i}$. Of course, it is important to note that consecutive applications of twisting maps are compatible with the multiplications, namely the following diagrams commute.
\begin{center}
\begin{tikzcd}
B \otimes A^{\otimes i}  \arrow{r}{\tau_{i,A}} \arrow[swap]{d}{1^{\otimes (i-1)}\otimes \nabla_A} & A^{\otimes i}\otimes B \arrow{d}{\nabla_A \otimes 1^{\otimes (i-1)}}\\
B \otimes A^{\otimes (i-1)}  \arrow{r}{\tau_{i-1,A}} & A^{\otimes (i-1)}\otimes B
\end{tikzcd} \quad \begin{tikzcd}
B^{\otimes i} \otimes A  \arrow{r}{\tau_{i,A}} \arrow[swap]{d}{\nabla_B\otimes 1^{\otimes (i-1)}} & A\otimes B^{\otimes i} \arrow{d}{1^{\otimes (i-1)}\otimes\nabla_B}\\
B^{\otimes (i-1)} \otimes A  \arrow{r}{\tau_{i-1,A}} & A\otimes B^{\otimes (i-1)}
\end{tikzcd}
\end{center}
More concisely, for all $i,j\in \N$ we have the following commutative diagrams.
\begin{center}
\begin{tikzcd}
B^{\otimes i} \otimes A^{\otimes j}  \arrow{rr}{\tau_{i,j}} \arrow[swap]{d}{\nabla_B\otimes 1^{\otimes (i+j-2)}\otimes \nabla_A} & & A^{\otimes j}\otimes B^{\otimes i} \arrow{d}{\nabla_A\otimes 1^{\otimes (i+j-2)}\otimes \nabla_B}\\
B^{\otimes (i-1)} \otimes A^{\otimes (j-1)} \arrow{rr}{\tau_{i-1,j-1}} & & A^{\otimes (j-1)}\otimes B^{\otimes (i-1)}
\end{tikzcd}
\end{center}
When diagrams \eqref{diag:tau-commutes-cmA} and \eqref{diag:tau-commutes-cmB} commute then $\tau$ is compatible with the comultiplications, and the corresponding diagrams also commute.

\begin{theo}\label{theo:twist-separable}
Let $(A,\nabla_A,\eta_A)$ and $(B,\nabla_B,\eta_B)$ be separable unital associative $k$-algebras, and denote by $\Gamma_A:A\to A\otimes A$ and $\Gamma_B:B\to B\otimes B$ the right inverses of the respective multiplications. Then, the following diagrams commute if and only if $(A\otimes_{\tau} B,\nabla_{A\otimes_{\tau} B},\eta_{A\otimes_{\tau} B})$ is a separable unital associative $k$-algebra and $\Gamma_{A\otimes_{\tau} B} : A\otimes_{\tau} B \to (A\otimes_{\tau} B) \otimes (A\otimes_{\tau} B)$ given by $\Gamma_{A\otimes_{\tau} B} = (1\otimes \tau^{-1} \otimes 1)(\Gamma_A \otimes\Gamma_B)$ is a right inverse of its multiplication.
\begin{equation}\label{diag:tau-commutes-sectionA}
\begin{tikzcd}
B\otimes A \arrow{rr}{\tau} \arrow[swap]{d}{1\otimes \Gamma_A} & & A\otimes B \arrow{d}{\Gamma_A\otimes 1} \\
B\otimes A\otimes A  \arrow{r}{\tau \otimes 1} & A\otimes B\otimes A \arrow{r}{1 \otimes \tau} & A\otimes A\otimes B
\end{tikzcd}
\end{equation}
\begin{equation}\label{diag:tau-commutes-sectionB}
\begin{tikzcd}
B\otimes A \arrow{rr}{\tau} \arrow[swap]{d}{\Gamma_B \otimes 1} & & A\otimes B \arrow{d}{1\otimes \Gamma_B} \\
B\otimes B\otimes A  \arrow{r}{1\otimes\tau} & B\otimes A\otimes B \arrow{r}{\tau\otimes 1} & A\otimes B\otimes B
\end{tikzcd}
\end{equation}
\end{theo}
\begin{proof}
Clearly $(1\otimes \tau^{-1} \otimes 1)(\Gamma_A \otimes\Gamma_B)$ is a right inverse of $\nabla_{A\otimes_{\tau} B}$ since the diagram
\begin{center}
\begin{tikzcd}
A B \arrow{r}{\Gamma_{A} 1} \arrow[swap]{rd}{1} & AAB \arrow{r}{1 \Gamma_{B}} \arrow{rd}{1} \arrow{d}{\nabla_{A} 1} & AABB \arrow{r}{1\tau^{-1} 1} \arrow{rd}{1} \arrow{d}{11\nabla_{B}} & ABAB \arrow{d}{1 \tau 1} \\
 & AB & AAB \arrow{l}{\nabla_{A} 1} & AABB \arrow{l}{11\nabla_{B}}
\end{tikzcd}
\end{center}
commutes. Now $\Gamma_{A\otimes_{\tau} B}$ is an $A\otimes_{\tau} B$ bimodule morphism if and only if the diagram
\begin{center}
\begin{tikzcd}
AB AB AB \arrow{rr}{11\Gamma_{A} \Gamma_{B} 11} \arrow[swap]{d}{1\tau 111} & & ABAABBAB \arrow{rr}{111\tau^{-1} 111} \arrow{dr}{1\tau 11\tau 1} & & ABABABAB \arrow[swap]{dd}{1\tau 11\tau 1}\\
AABBAB \arrow[swap]{d}{\nabla_A \nabla_B 11} & & AAABABBB \arrow[swap]{d}{111\tau 111} & AABABABB \arrow[swap]{l}{11\tau \tau 11} \arrow{rd}{111\tau^{-1} 111} & \\
ABAB \arrow[swap]{d}{1\tau 1} \arrow[draw = none]{rr}{\fbox{?}} & & AAAABBBB \arrow{rr}{11\tau^{-1}_{2,2} 11} \arrow[swap]{ddd}{\nabla_A \nabla_A \nabla_B \nabla_B} \arrow{dr}{\nabla_A 1111 \nabla_B} & & AABBAABB \arrow[swap]{ddd}{\nabla_A \nabla_B \nabla_A \nabla_B} \arrow[swap, near start]{ddl}{\nabla_A 1111 \nabla_B}\\
AABB \arrow[swap]{dd}{\nabla_A \nabla_B} & & & AAABBB \arrow[near start]{d}{1\tau^{-1}_{2,2}1} \arrow[swap]{ddl}{1 \nabla_A \nabla_B 1} & \\
 & & & ABBAAB \arrow{dr}{1 \nabla_B \nabla_A 1} & \\
AB \arrow{rr}{\Gamma_A \Gamma_B} & & AABB \arrow{rr}{1\tau^{-1} 1} & & ABAB
\end{tikzcd}
\end{center}
commutes, if and only if the diagram
\begin{center}
\begin{tikzcd}
AB AB AB \arrow{rr}{11\Gamma_{A} 111} \arrow[swap]{dd}{1\tau 111} \arrow{ddrr}{111\tau 1} & & ABAABAB \arrow{rr}{1111\Gamma_{B} 11} \arrow{ddr}{1111\tau 1} & & ABAABBAB \arrow[swap]{d}{11111\tau 1}\\
 & & & & ABAABABB \arrow[swap]{d}{1111\tau 11} \\
AABBAB \arrow[swap]{dd}{\nabla_A \nabla_B 11} \arrow{ddrr}{111\tau 1} & & ABAABB \arrow{r}{11 \Gamma_A 111} \arrow{dd}{1\tau 111} & ABAAABB \arrow{d}{1\tau 1111} \arrow{r}{11111 \Gamma_B 1} & ABAAABBB \arrow[swap]{d}{1\tau 11111}\\
 & & & AABAABB \arrow{d}{11\tau 111} \arrow{r}{11111 \Gamma_B 1} & AABAABBB \arrow[swap]{d}{11\tau 1111} \\
ABAB \arrow[swap]{d}{1\tau 1} & & AABABB \arrow{r}{1 \Gamma_A 1111} \arrow{d}{11\tau 11} & AAABABB \arrow{r}{11111 \Gamma_B 1} & AAABABBB \arrow[swap]{d}{111\tau 111} \\
AABB \arrow[swap]{d}{\nabla_A \nabla_B} \arrow[draw = none]{rr}{\fbox{?}} & & AAABBB \arrow{rr}{1\Gamma_A 11 \Gamma_B 1} & & AAAABBBB \arrow[swap]{d}{\nabla_A \nabla_A \nabla_B \nabla_B} \\
AB \arrow{rrrr}{\Gamma_A \Gamma_B} & & & & AABB
\end{tikzcd}
\end{center}
commutes, if and only if the diagram
\begin{center}
\begin{tikzcd}
AABBAB \arrow[swap]{d}{\nabla_A 1111} \arrow{r}{111 \tau 1} & AABABB \arrow{rr}{11 \tau 11} \arrow{d}{\nabla_A 1111} & & AAABBB \arrow{dd}{1\Gamma_A 11 \Gamma_B 1} \arrow{dl}{\nabla_A 1111} \\
ABBAB \arrow[swap]{d}{1 \nabla_B 11} \arrow{r}{11 \tau 1} & ABABB \arrow{r}{1 \tau 11} & AABBB \arrow[swap]{dl}{11 \nabla_B 1} & \\
ABAB \arrow{r}{1 \tau 1} & AABB \arrow{d}{\nabla_A 11} \arrow{r}{11 \nabla_B} & AAB \arrow{d}{\nabla_A 1} \arrow[draw = none]{r}{\fbox{?}} & AAAABBBB \arrow{d}{\nabla_A \nabla_A \nabla_B \nabla_B} \\
 & ABB \arrow{r}{1 \nabla_B} & AB \arrow{r}{\Gamma_A \Gamma_B} & AABB
\end{tikzcd}
\end{center}
commutes, if and only if the diagram
\begin{center}
\begin{tikzcd}
AAABBB \arrow[swap]{dd}{1111\Gamma_B 1} \arrow{r}{\nabla_A 1111} \arrow{dr}{1111 \nabla_B} & AABBB \arrow{r}{11 \nabla_B 1} \arrow{dr}{111 \nabla_B} & AABB \arrow{r}{11 \nabla_B} & AAB \arrow{dd}{\nabla_A 1}\\
 & AAABB \arrow{r}{\nabla_A 111} \arrow{dr}{111 \nabla_B} & AABB \arrow{ur}{11 \nabla_B} & \\
AAABBBB \arrow{r}{111 \nabla_B \nabla_B} \arrow[swap]{dd}{1 \Gamma_A 11111} & AAABB \arrow{d}{1\Gamma_A 111} & AAAB \arrow[swap]{uur}{\nabla_A 11} \arrow[swap]{l}{111 \Gamma_B} \arrow{d}{1\Gamma_A 11} & AB \arrow{d}{\Gamma_A 1}\\
 & AAAABB \arrow{drr}{\nabla_A \nabla_A 11} & AAAAB \arrow[swap]{l}{1111 \Gamma_B} \arrow{r}{\nabla_A \nabla_A 11} & AAB \arrow{d}{11 \Gamma_B}\\
AAAABBBB \arrow{ur}{1111 \nabla_B \nabla_B} \arrow{rrr}{\nabla_A \nabla_A \nabla_B \nabla_B} & & & AABB
\end{tikzcd}
\end{center}
commutes. This finishes the proof.
\end{proof}

\section{Twisted tensor product of bialgebras}\label{sec:ttp-bialgebras}

In this section, we determine when a twisted tensor product of bialgebras inherits a bialgebra structure from the twisting map. As an application, we obtain that twisted tensor products of Hopf algebras inherit the Hopf algebra structure if and only if the twisting map is trivial.

Let $(A,\nabla_A,\eta_A,\Delta_A,\epsilon_A)$ and $(B,\nabla_B,\eta_B,\Delta_B,\epsilon_B)$ be $k$-bialgebras for the remainder of this section, and let $\tau:B\otimes A\rightarrow A\otimes B$ be a twisting map. Now $(A\otimes_{\tau} B,\nabla_{A\otimes_{\tau} B},\eta_{A\otimes_{\tau} B})$ is a unital associative algebra by Proposition~\ref{prop:AtB-unital-associative}, but $(A\otimes_{\tau} B,\Delta_{A\otimes_{\tau} B},\epsilon_{A\otimes_{\tau} B})$ is not necessarily a counital coassociative coalgebra. The compatibility of $\tau$ with the counit $\epsilon_B$ is established by the commutative diagram
\begin{center}
\begin{tikzcd}
AB \arrow[swap]{dd}{1 \epsilon_B} & & & B A \arrow{dd}{\epsilon_B 1} \arrow{lll}{\tau} \\
 & Ak \arrow{r}{\cong} \arrow{ul}{1\eta_B} \arrow{dl}{11} & kA \arrow{ur}{\eta_B 1} \arrow{dr}{11} & \\
Ak \arrow{rrr}{\cong} & & & kA
\end{tikzcd}
\end{center}
and the compatibility of $\tau$ with the counit $\epsilon_A$ follows similarly. Hence diagrams \eqref{diag:tau-counits} commute so by Proposition~\ref{prop:AtB-counital-coassociative} the only condition imposing a restriction on $\tau$ to obtain the desired coalgebra structure is its compatibility with the comultiplications $\Delta_B$ and $\Delta_A$. The following examples show this cannot be guaranteed in general.

\begin{exam}[Jordan plane]\label{exam:kx-twist-not-cotwist-Jordan}
Let $k[x]$ and $k[y]$ be polynomial rings in one variable with the bialgebra structure of Example~\ref{exam:bialgebra}. Consider the twisting map $\tau:k[y]\otimes k[x]\rightarrow k[x]\otimes k[y]$ given by linearly extending $\tau(y\otimes x) = x\otimes y + x^2\otimes 1$. Observe that
\begin{equation*}
k[x]\otimes_{\tau}k[y]\cong k\langle x,y\rangle/(x y - y x + x^2)
\end{equation*}
is the Jordan plane. The twisting map is not compatible with the comultiplications, namely diagram \eqref{diag:tau-comultiplications} does not commute.

This example can be extended effortlessly to $2n\in\N$ variables, obtaining
\begin{equation*}
\frac{k\langle x_1,\dots, x_n,y_1,\dots, y_n\rangle}{\left(x_i x_j - x_j x_i,\, y_i y_j - y_j y_i,\, x_i y_j - y_j x_i + \delta_{i,j} x_i^2 \right)_{i,j\in\{1,\dots,n\}}}.
\end{equation*}
Again, the twisting map is not compatible with the comultiplications.
\end{exam}

\begin{exam}[Weyl algebra]\label{exam:kx-twist-not-cotwist-Weyl}
Let $k[x]$ and $k[y]$ be polynomial rings in one variable with the bialgebra structure of Example~\ref{exam:bialgebra}. Consider the twisting map $\tau:k[y]\otimes k[x]\rightarrow k[x]\otimes k[y]$ given by linearly extending $\tau(y\otimes x) = x\otimes y - 1\otimes 1$. Now
\begin{equation*}
k[x]\otimes_{\tau}k[y]\cong k\langle x,y\rangle/(x y - y x - 1)
\end{equation*}
is the Weyl algebra, and the twisting map is also not compatible with the comultiplications. We can again extend this to
\begin{equation*}
\frac{k\langle x_1,\dots, x_n,y_1,\dots, y_n\rangle}{\left(x_i x_j - x_j x_i,\, y_i y_j - y_j y_i,\, x_i y_j - y_j x_i - \delta_{i,j} \right)_{i,j\in\{1,\dots,n\}}},
\end{equation*}
where once more the twisting map is not compatible with the comultiplications.
\end{exam}

\begin{exam}[Quantum enveloping algebra]\label{exam:kx-twist-not-cotwist-Uqsl2}
Let $q\in k^{\times}$, let $\Uqsl$ be the free $k$-algebra generated by $E$, $F$, $K$, and $K^{-1}$ subject to the relations
\begin{equation*}
KK^{-1} = 1 = K^{-1}K, \; KE = q^2EK, \; KF = q^{-2}FK, \; EF - FE = \frac{K-K^{-1}}{q-q^{-1}}.
\end{equation*}
This is a quantized enveloping algebra whose Borel subalgebra is the free $k$-algebra $\Uqb$ generated by $E$, $K$, and $K^{-1}$, subject to the same relations. Consider the map $\tau: \Uqb \otimes k[F] \to k[F] \otimes \Uqb$ given by extending
\begin{equation*}
\tau(K\otimes F) = q^{-2}F\otimes K,\quad \tau(E\otimes F) = F\otimes E - \frac{1 \otimes K - 1\otimes K^{-1}}{q-q^{-1}}.
\end{equation*}
This gives a twisting map making $k[F]\otimes_\tau \Uqb$ isomorphic to $\Uqsl$ as $k$-algebras. Considering $k[F]$ as a polynomial ring in one variable with the bialgebra structure of Example~\ref{exam:bialgebra}, and $\Uqb$ with its usual coalgebra structure, whose comultiplication is given by extending
\begin{equation*}
\Delta(E) = 1\otimes E + E\otimes K, \qquad \Delta(K) = K\otimes K,
\end{equation*}
then the twisting map is not compatible with the comultiplications.
\end{exam}

Consequently, we have to require that $\tau$ makes diagram \eqref{diag:tau-comultiplications} commute. This requirement is equivalent to $(A\otimes_{\tau} B,\Delta_{A\otimes_{\tau} B},\epsilon_{A\otimes_{\tau} B})$ being a counital coassociative coalgebra. Independently of this, we always have the commutative diagrams
\begin{center}
\begin{equation*}
\begin{tikzcd}
ABAB \arrow[swap]{d}{1 \epsilon_B 1 \epsilon_B} \arrow{r}{1\tau 1} & AABB \arrow{r}{\nabla_A \nabla_B} \arrow{d}{11 \epsilon_B \epsilon_B} & AB \arrow{dd}{\epsilon_A \epsilon_B} \\
AkAk \arrow[swap]{d}{\epsilon_A 1 \epsilon_A 1} \arrow{r}{\cong} & AAkk \arrow{d}{\epsilon_A \epsilon_A 11} & \\
kkkk \arrow{r}{1111} & kkkk \arrow{r}{\cong} & kk
\end{tikzcd} \quad
\begin{tikzcd}
kk \arrow{r}{\cong} \arrow[swap]{dd}{\eta_A\eta_B} & kkkk \arrow{r}{1111} \arrow{d}{\eta_A\eta_A11} & kkkk \arrow{d}{\eta_A1\eta_A1}\\
& AAkk \arrow{r}{\cong} \arrow{d}{11\eta_B\eta_B} & AkAk \arrow{d}{1\eta_B1\eta_B} \\
AB \arrow{r}{\Delta_A\Delta_B} & AABB \arrow{r}{1\tau^{-1}1} & ABAB
\end{tikzcd}
\end{equation*}
\end{center}
yielding the compatibility of $\nabla_{A\otimes_{\tau} B}$ with $\epsilon_{A\otimes_{\tau} B}$ and the compatibility of $\Delta_{A\otimes_{\tau} B}$ with $\eta_{A\otimes_{\tau} B}$. Moreover the commutative diagram
\begin{center}
\begin{tikzcd}
kk \arrow{r}{\eta_A 1} \arrow{d}{11} & Ak \arrow{r}{1 \eta_B} \arrow{d}{11} & AB \arrow{d}{11} \\
kk & Ak \arrow[swap]{l}{\epsilon_A 1} & AB \arrow[swap]{l}{1 \epsilon_B}
\end{tikzcd}
\end{center}
establishes the compatibility of $\eta_{A\otimes_{\tau} B}$ with $\epsilon_{A\otimes_{\tau} B}$. However, this still does not guarantee that $A\otimes_{\tau} B$ is a bialgebra.

\begin{exam}[Quantum plane or quantum affine space]\label{exam:kx-twist-not-bialg-quantum-plane}
Let $k[x]$ and $k[y]$ be polynomial rings in one variable with the bialgebra structure of Example~\ref{exam:bialgebra}. Consider the twisting map $\tau:k[y]\otimes k[x]\rightarrow k[x]\otimes k[y]$ given by linearly extending $\tau(y\otimes x) = q x\otimes y$ for some non-zero $q\in k$. Then
\begin{equation*}
k[x]\otimes_{\tau}k[y]\cong k_q[x,y] = k\langle x,y\rangle/(q x y - y x)
\end{equation*}
is the quantum plane. The twisting map $\tau$ is compatible with the comultiplications of $k[x]$ and $k[y]$, namely diagram \eqref{diag:tau-comultiplications} commutes. However, $k[x]\otimes_{\tau}k[y]$ is not a bialgebra. Following Examples~\ref{exam:kx-twist-not-cotwist-Jordan} and~\ref{exam:kx-twist-not-cotwist-Weyl}, we can extend this to $n\in\mathbb{N}$ variables. Given $\mathbf{q} = (q_{ij}) \in M_{n}(k^{\times})$ a square matrix with non-zero entries such that $q_{ii} = 1$ and $q_{ij} q_{ji} = 1$ for all $1\leq i,j\leq n$, set
\begin{equation*}
k_{\mathbf{q}}[x_1,\dots,x_n] = \frac{k\langle x_1,\dots, x_n\rangle}{\left(x_i x_j - q_{ij} x_j x_i \right)_{i,j\in\{1,\dots,n\}}}.
\end{equation*}
Again, the twisting maps are compatible with the comultiplications, but the induced algebra and coalgebra structures on $k_{\mathbf{q}}[x_1,\dots,x_n]$ do not give a bialgebra.
\end{exam}

For $A\otimes_{\tau} B$ to be a bialgebra we have to impose the compatibility of the multiplication and comultiplication. Requiring this completely determines $\tau$.

\begin{lemm}\label{lemm:auxiliary-commutative-diagrams-AB}
Let $A$ and $B$ be $k$-bialgebras, let $\tau:B\otimes A\rightarrow A\otimes B$ be a twisting map. Then the following diagrams commute.
\begin{center}
\begin{tikzcd}
AkkB \arrow{r}{1111} \arrow[swap]{d}{1 \eta_B \eta_A 1} & AkkB \arrow[swap]{d}{1 \eta_A \eta_B 1}\\
ABAB \arrow{r}{1\tau 1} & AABB
\end{tikzcd} \quad
\begin{tikzcd}
AAkkkkBB \arrow{rr}{\cong} \arrow[swap]{d}{11 \eta_B \eta_B \eta_A \eta_A 11} & & AkAkkBkB \arrow{d}{1 \eta_B 1\eta_B \eta_A 1\eta_A 1}\\
AABBAABB \arrow{rr}{1\tau^{-1} 11\tau^{-1} 1} & & ABABABAB
\end{tikzcd}
\end{center}
\begin{center}
\begin{tikzcd}
AABB \arrow{r}{\sigma_{23}} \arrow[swap]{d}{\cong} & ABAB \arrow{d}{\cong}\\
AkAkkBkB \arrow[swap]{d}{1 \eta_B 1\eta_B \eta_A 1\eta_A 1} & AkkBAkkB \arrow{d}{1 \eta_B \eta_A 11 \eta_B \eta_A 1}\\
ABABABAB \arrow{r}{\sigma_{35}\circ \sigma_{46}} & ABABABAB
\end{tikzcd}
\end{center}
\end{lemm}
\begin{proof}
The diagram
\begin{center}
\begin{tikzcd}
AkkB \arrow{r}{1\eta_B 11} \arrow{d}{1111} & ABkB \arrow{r}{11\eta_A 1} \arrow{d}{\cong} & ABAB \arrow{d}{1\tau 1}\\
AkkB \arrow{r}{11\eta_B 1} & AkBB \arrow{r}{1\eta_A 11} & AABB
\end{tikzcd}
\end{center}
commutes. The diagram 
\begin{center}
\begin{tikzcd}
AAkkkkBB \arrow{rr}{111\eta_B\eta_A 111} \arrow{d}{\cong} & & AAkBAkBB \arrow{d}{\cong} \arrow{rr}{11\eta_B 11\eta_A 11} & & AABBAABB \arrow{d}{1\tau^{-1} 11\tau^{-1} 1}\\
AkAkkBkB \arrow{rr}{111\eta_B\eta_A 111} & & AkABABkB \arrow{rr}{1\eta_B 1111\eta_A 1} & & ABABABAB
\end{tikzcd}
\end{center}
commutes,  and the following diagram commutes.
\begin{center}
\begin{tikzcd}
AABB \arrow{r}{\cong} \arrow{d}{\sigma_{23}} & AkAkkBkB \arrow{rr}{1\eta_B 1111 \eta_A 1} \arrow{d}{\sigma_{35}\circ\sigma_{46}} & & ABAkkBAB \arrow{rr}{111\eta_B \eta_A 111} \arrow{d}{\sigma_{35}\circ\sigma_{46}} & & ABABABAB \arrow{d}{\sigma_{35}\circ \sigma_{46}}\\
ABAB \arrow{r}{\cong} & AkkBAkkB \arrow{rr}{1\eta_B 1111 \eta_A 1} & & ABkBAkAB \arrow{rr}{11\eta_A 11 \eta_B 11} & & ABABABAB
\end{tikzcd}
\end{center}
\end{proof}

\begin{lemm}\label{lemm:auxiliary-commutative-diagrams-BA}
Let $A$ and $B$ be $k$-bialgebras, let $\tau:B\otimes A\rightarrow A\otimes B$ be a twisting map. Then the following diagrams commute.
\begin{center}
\begin{tikzcd}
BBAA \arrow{r}{\sigma_{23}} \arrow[swap]{d}{\cong} & BABA \arrow{d}{\cong}\\
kBkBAkAk \arrow[swap]{d}{\eta_A 1\eta_A 11 \eta_B 1\eta_B} & kBAkkBAk \arrow[swap]{d}{\eta_A 11 \eta_B \eta_A 11 \eta_B}\\
ABABABAB \arrow{r}{\sigma_{35}\circ \sigma_{46}} & ABABABAB
\end{tikzcd}\quad
\begin{tikzcd}
BBAA \arrow{rr}{1111} \arrow[swap]{d}{\cong} & & BBAA \arrow{d}{\cong}\\
kkBBAAkk \arrow{rr}{\cong} \arrow{d}{\eta_A \eta_A 1111 \eta_B\eta_B} & & kBkBAkAk \arrow[swap]{d}{\eta_A 1 \eta_A 11 \eta_B 1 \eta_B}\\
AABBAABB \arrow{rr}{1\tau^{-1} 11\tau^{-1} 1} & & ABABABAB
\end{tikzcd}
\end{center}
\begin{center}
\begin{tikzcd}
kBAk \arrow{r}{1\tau 1} \arrow[swap]{dd}{\eta_A 11 \eta_B} & kABk \arrow{d}{\cong}\\
 & AB\\
ABAB \arrow{r}{1\tau 1} & AABB \arrow[swap]{u}{\nabla_A \nabla_B}
\end{tikzcd} \quad
\end{center}
\end{lemm}
\begin{proof}
The diagram
\begin{center}
\begin{tikzcd}
BBAA \arrow{r}{\cong} \arrow{d}{\sigma_{23}} & kBkBAkAk \arrow{rr}{\eta_A 111111 \eta_B} \arrow{d}{\sigma_{35}\circ\sigma_{46}} & & ABkBAkAB \arrow{rr}{11\eta_A 11\eta_B 11} \arrow{d}{\sigma_{35}\circ\sigma_{46}}  & & ABABABAB \arrow{d}{\sigma_{35}\circ \sigma_{46}}\\
BABA \arrow{r}{\cong} & kBAkkBAk \arrow{rr}{\eta_A 111111 \eta_B} & & ABAkkBAB \arrow{rr}{111 \eta_B \eta_A 111} & & ABABABAB
\end{tikzcd}
\end{center}
commutes. the diagram
\begin{center}
\begin{tikzcd}
kkBBAAkk \arrow{d}{\cong} \arrow{rr}{\eta_A 111111 \eta_B} & & AkBBAAkB \arrow{rr}{1 \eta_A 1111 \eta_B 1} \arrow{d}{\cong} & & AABBAABB \arrow{d}{1\tau^{-1} 11 \tau^{-1} 1} \\
kBkBAkAk \arrow{rr}{\eta_A 111111 \eta_B} & & BBkBAkAB \arrow{rr}{11 \eta_A 11 \eta_B 11} & & ABABABAB
\end{tikzcd}
\end{center}
commutes,  and the following diagram commutes.
\begin{center}
\begin{tikzcd}
kBAk \arrow{r}{1\tau 1} \arrow{d}{\eta_A 11 \eta_B} & kABk \arrow{r}{\cong} \arrow[swap]{d}{\eta_A 11 \eta_B} & AB\\
ABAB \arrow{r}{1\tau 1} & AABB \arrow[swap]{ur}{\nabla_A \nabla_B}
\end{tikzcd}
\end{center}
\end{proof}

\begin{lemm}\label{lemm:cmA-cmB-tau-sigma}
Let $A$ and $B$ be $k$-bialgebras, let $\tau:B\otimes A\rightarrow A\otimes B$ be a twisting map. If $(A\otimes_{\tau} B,\nabla_{A\otimes_{\tau} B},\eta_{A\otimes_{\tau} B},\Delta_{A\otimes_{\tau} B},\epsilon_{A\otimes_{\tau} B})$ is a $k$-bialgebra then the following diagrams commute.
\begin{equation*}
\begin{tikzcd}
AB \arrow{r}{\Delta_A\Delta_B} \arrow{d}[swap]{\Delta_A\Delta_B} & AABB \arrow{d}{1\tau^{-1}1} \\
AABB \arrow{r}[swap]{\sigma_{23}} & ABAB
\end{tikzcd} \qquad
\begin{tikzcd}
BA \arrow{r}{\Delta_B\Delta_A} \arrow{d}[swap]{\Delta_B\Delta_A} & BBAA \arrow{d}{1\tau 1} \\
BBAA \arrow{r}[swap]{\sigma_{23}} & BABA
\end{tikzcd} 
\end{equation*}
\end{lemm}
\begin{proof}
In this proof we slightly abuse notation by omitting the field $k$ to simplify the diagrams. Observe that if $A\otimes_\tau B$ is a bialgebra then it is a counital coassociative coalgebra, so diagram~\eqref{diag:tau-comultiplications} commutes by Proposition~\ref{prop:AtB-counital-coassociative}. The outside of the following diagram commutes because $A\otimes_{\tau} B$ is a $k$-bialgebra. The diagrams indicated by a $\circlearrowleft$ commute by Lemma~\ref{lemm:auxiliary-commutative-diagrams-AB}. Thus, the outer diagrams commute, and so does the central square. 
\begin{center}
\begin{tikzcd}
ABAB \arrow{rr}{1\tau1} \arrow[near end]{dd}{\Delta_A\Delta_B\Delta_A\Delta_B} & \phantom{x} \arrow[phantom,pos=0.2]{dr}{\circlearrowleft} & AABB \arrow{r}{\nabla_A\nabla_B} & AB \arrow{rr}{\Delta_A\Delta_B} & & AABB \arrow[swap]{dddd}{1\tau^{-1}1}\\
 & \phantom{x} & AB \arrow[swap]{r}{\Delta_A\Delta_B} \arrow{d}{\Delta_A\Delta_B} \arrow{ull}{1\eta_B\eta_A1} \arrow{u}{1\eta_A\eta_B 1} \arrow{ur}{11} & AABB \arrow{d}{1\tau^{-1}1} \arrow{urr}{1111} & & \\
AABBAABB \arrow[near start]{dd}{1\tau^{-1}11\tau^{-1}1} \arrow[phantom,pos=0.4]{ddrr}{\circlearrowleft} & \phantom{x} & AABB \arrow{r}{\sigma_{23}} \arrow[swap]{ll}{11\eta_B\eta_B\eta_A\eta_A11} \arrow{ddll}{1\eta_B1\eta_B\eta_A1\eta_A1} \arrow[phantom]{dd}{\circlearrowleft} & ABAB \arrow[near end]{ddl}{1\eta_B\eta_A 11 \eta_B\eta_A 1} \arrow[near end]{dd}{1\eta_A\eta_B11\eta_A\eta_B1} \arrow{ddrr}{1111} & & \phantom{x} \\
 & & & \phantom{x} & \\
ABABABAB \arrow[swap]{rr}{\sigma_{35}\circ \sigma_{46}} & & ABABABAB \arrow{r}[swap]{1\tau11\tau1} & AABBAABB \arrow[swap]{rr}{\nabla_A\nabla_B\nabla_A\nabla_B} & & ABAB\\
\end{tikzcd}
\end{center}

Similarly, the parts of the following diagram indicated by a $\circlearrowleft$ commute by Lemma~\ref{lemm:auxiliary-commutative-diagrams-BA}, and therefore, the inner square commutes.
\begin{center}
\begin{tikzcd}
ABAB \arrow{rr}{1\tau1} \arrow[near end]{dd}{\Delta_A\Delta_B\Delta_A\Delta_B} &  & AABB \arrow{r}{\nabla_A\nabla_B} \arrow[phantom]{d}{\circlearrowleft} & AB \arrow{rr}{\Delta_A\Delta_B} & & AABB \arrow[swap]{ddd}{1\tau^{-1} 1}\\
 & & BA \arrow[swap]{r}{\Delta_B\Delta_A} \arrow{d}{\Delta_B\Delta_A} \arrow{ull}{\eta_A11\eta_B} \arrow{ur}{\tau}  & BBAA \arrow{d}{1\tau1} & & \\
AABBAABB \arrow[near start]{d}{1\tau^{-1}11\tau^{-1}1} & \phantom{x} \arrow[phantom,pos=0.15]{d}{\circlearrowleft}& BBAA \arrow{r}{\sigma_{23}} \arrow[swap]{ll}{\eta_A\eta_A 1111 \eta_B\eta_B} \arrow{dll}{\eta_A1\eta_A11\eta_B1\eta_B} \arrow[phantom,pos=0.3]{d}{\circlearrowleft} & BABA \arrow{dl}{\eta_A11\eta_B\eta_A 11\eta_B} \arrow{drr}{\tau\tau} & & \\
ABABABAB \arrow[swap]{rr}{\sigma_{35}\circ \sigma_{46}} & \phantom{x} & ABABABAB \arrow[swap]{r}{1\tau11\tau1} & AABBAABB \arrow[swap]{rr}{\nabla_A\nabla_B\nabla_A\nabla_B} & & ABAB\\
\end{tikzcd}
\end{center}
\end{proof}

\begin{theo}\label{theo:bialgebra-twist-iff-cotwist}
Let $A$ and $B$ be $k$-bialgebras, let $\tau:B\otimes A\rightarrow A\otimes B$ be a twisting map. Then $(A\otimes_{\tau} B,\nabla_{A\otimes_{\tau} B},\eta_{A\otimes_{\tau} B},\Delta_{A\otimes_{\tau} B},\epsilon_{A\otimes_{\tau} B})$ is a $k$-bialgebra if and only if $\tau$ is trivial.
\end{theo}
\begin{proof}
Suppose that $\tau$ is trivial, then $A\otimes_\tau B \cong A\otimes B$ and the claim is clear. Suppose that $A\otimes_\tau B$ is a bialgebra, in particular it is a counital coassociative coalgebra, so diagram~\eqref{diag:tau-comultiplications} commutes by Proposition~\ref{prop:AtB-counital-coassociative}. Note then that the following diagram commutes.
\begin{center}
\begin{tikzcd}
BABA \arrow{rrr}{\tau\tau} \arrow[swap]{d}{\sigma_{23}} \arrow{dr}{111\epsilon_A} & & & ABAB \arrow{r}{\sigma_{23}} \arrow{d}{11\epsilon_A 1}  & AABB \arrow{d}{1\epsilon_A11} \\
BBAA \arrow[swap]{d}{111\epsilon_A} & BABk \arrow{r}{\tau 11} \arrow{dr}{\epsilon_B111} & ABBk \arrow{r}{\cong} & ABkB \arrow{r}{\sigma_{23}} \arrow{d}{1\epsilon_B 11} & AkBB \arrow{d}{11\epsilon_B 1}\\
BBAk \arrow{r}{\epsilon_B 111} \arrow[swap]{ur}{\sigma_{23}} & kBAk \arrow{r}{\sigma_{23}} & kABk \arrow{r}{\cong} & AkkB \arrow{r}{1111} & AkkB
\end{tikzcd}
\end{center}
Then, the outside of the following diagram commutes, and the parts of the diagram that commute by Lemma~\ref{lemm:cmA-cmB-tau-sigma} have been marked by a $\circlearrowleft$.
\begin{center}
\begin{tikzcd}
BABA \arrow[phantom]{ddr}{\circlearrowleft} & BBAA \arrow{l}[swap]{1\tau1} \arrow{r}{\tau_{2,2}} & AABB \arrow{r}{1\tau^{-1}1} \arrow[phantom]{ddr}{\circlearrowleft} & ABAB \\
& BA \arrow{u}{\Delta_B\Delta_A} \arrow{r}{\tau} \arrow[swap]{d}{11} & AB \arrow[swap]{u}{\Delta_A\Delta_B} \arrow{d}{11} & \\
BBAA \arrow{uu}{\sigma_{23}}\arrow{d}[swap]{\epsilon_B 11 \epsilon_A} & BA \arrow{l}[swap]{\Delta_B\Delta_A} \arrow[swap]{d}{11} \arrow{r}{\sigma_{12}} & AB \arrow{d}{11} \arrow{r}{\Delta_A\Delta_B}  & AABB \arrow{d}{1\epsilon_A \epsilon_B 1} \arrow{uu}[swap]{\sigma_{23}}\\
kBAk \arrow{r}{\cong} & BA \arrow{r}{\sigma_{12}} & AB \arrow{r}{\cong} & AkkB \\
\end{tikzcd}
\end{center}
Thus, the center square of the above diagram also commutes, proving that $\tau$ is trivial.
\end{proof}

Alas, it is now clear that non-trivial Hopf algebras will not come from twisted tensor products with the counit and comultiplication induced by the twisting map.

\begin{coro}\label{coro:trivial-Hopf}
Let $A$ and $B$ be Hopf algebras, let $\tau:B\otimes A\rightarrow A\otimes B$ be a twisting map such that diagram~\eqref{diag:tau-comultiplications} commutes. If $A\otimes_{\tau} B$ is a Hopf algebra, then it is the tensor product Hopf algebra $A\otimes B$.
\end{coro}

\begin{coro}\label{coro:no-inherited-Hopf}
Let $A$ and $B$ be Hopf algebras, let $\tau:B\otimes A\rightarrow A\otimes B$ be a non-trivial twisting map such that diagram~\eqref{diag:tau-comultiplications} commutes. Then $A\otimes_{\tau} B$ is not a Hopf algebra.
\end{coro}

\section{Twisted tensor product of Frobenius algebras}\label{sec:ttp-Frobenius}

In this section, we show that a twisted tensor product of Frobenius algebras always inherits a Frobenius algebra structure from the twisting map, and we explicitly give the induced pairing and co-pairing. Let $(A,\nabla_A,\eta_A,\Delta_A,\epsilon_A)$ and $(B,\nabla_B,\eta_B,\Delta_B,\epsilon_B)$ be Frobenius algebras over $k$ for the rest of the section.

\begin{theo}\label{theo:ttp-cttp-always-Frobenius}
Let $A$ and $B$ be Frobenius algebras over $k$, let $\tau:B\otimes A\rightarrow A\otimes B$ be a twisting map such that diagram~\eqref{diag:tau-comultiplications} commutes. Then $(A\otimes_{\tau} B,\nabla_{A\otimes_{\tau} B},\eta_{A\otimes_{\tau} B},\Delta_{A\otimes_{\tau} B},\epsilon_{A\otimes_{\tau} B})$ is a Frobenius algebra.
\end{theo}
\begin{proof}
The diagram
\begin{center}
\begin{tikzcd}
ABAB \arrow{rr}{1\tau 1} \arrow{dr}{1\Delta_{B}11} \arrow[swap]{dd}{\Delta_{A} \Delta_{B} 11} & & AABB \arrow{r}{\nabla_{A} \nabla_{B}} \arrow{d}{11 \Delta_{B} 1} & AB \arrow{dd}{\Delta_A \Delta_B}\\
 & ABBAB \arrow{r}{1\tau_{2,A} 1} \arrow[swap]{dl}{\Delta_{A} 1111} & AABBB \arrow{d}{\Delta_{A}1111} & \\
AABBAB \arrow[swap]{dd}{1 \tau^{-1} 111} \arrow{dr}{111 \tau 1} \arrow{rr}{11 \tau_{2,A} 1} & & AAABBB \arrow{dd}{1 \tau_{B,2}^{-1} 11} \arrow[draw = none]{r}{\fbox{?}} & AABB \arrow{dd}{1\tau^{-1} 1} \\
 & AABABB \arrow{ur}{11 \tau 11} & & \\
ABABAB \arrow{rr}{111\tau 1} & & ABAABB \arrow{r}{11\nabla_{A} \nabla_{B}} \arrow{ul}{1 \tau 111} & ABAB
\end{tikzcd}
\end{center}
commutes because the diagram
\begin{center}
\begin{tikzcd}
AABB \arrow{r}{11\nabla_{B}} \arrow{d}{11\Delta_{B} 1} & AAB \arrow{rr}{\nabla_{A} 1} \arrow{d}{11\Delta_{B}} \arrow{dr}{\nabla_{A} 1} & & AB \arrow{d}{\Delta_{A} 1}\\
AABBB \arrow{dd}{\Delta_{A} 1111} \arrow{dr}{\nabla_{A} 111} \arrow{r}{111\nabla_{B}} & AABB \arrow{dr}{\nabla_{A} 11} & AB \arrow{r}{\Delta_{A} 1} \arrow{d}{1\Delta_{B}} & AAB \arrow{dd}{11\Delta_{B}}\\
 & ABBB \arrow{r}{111\nabla_{B}} \arrow{d}{\Delta_{A} 111} & ABB \arrow{dr}{\Delta_{A} 11} & \\
AAABBB \arrow{r}{1\nabla_{A} 111} & AABBB \arrow{rr}{111\nabla_{B}} & & AABB \\
ABAABB \arrow{u}{1\tau_{B,2} 11} \arrow{r}{11\nabla_{A} 11} & ABABB \arrow{rr}{111\nabla_{B}} \arrow{u}{1\tau 1} & & ABAB \arrow{u}{1\tau 1}
\end{tikzcd}
\end{center}
commutes, giving one of the associativity conditions on $\nabla_{A\otimes_{\tau} B}$ and $\Delta_{A\otimes_{\tau} B}$. The remaining one follows analogously.
\end{proof}

This is an unexpected generalization of the fact that the tensor product of Frobenius algebras is a Frobenius algebra (see \cite[Theorem 2.1]{MR110735}).

\begin{rema}\label{rema:if-and-only-if}
Note that if $A\otimes_{\tau} B$ is a Frobenius algebra then in particular $A\otimes_{\tau} B$ is a counital coassociative coalgebra, whence diagram~\eqref{diag:tau-comultiplications} commutes because of Proposition~\ref{prop:AtB-counital-coassociative}. Namely, Theorem~\ref{theo:ttp-cttp-always-Frobenius} gives a necessary and sufficient condition.
\end{rema}

\begin{exam}[Skew group algebra]\label{exam:twist-not-cotwist-SG}
Let $G$ and $H$ be finite groups where $G$ acts on $H$ via $\varphi: G \to \mathrm{Aut}(H)$. The map $\tau: kG\otimes kH \to kH \otimes kG$ given by $\tau(g\otimes h) = \varphi(g)(h)\otimes g$ for $g \in G$, $h\in H$ extends to a twisting map giving $kH \otimes_{\tau} kG \cong k(H\rtimes_\varphi G)$ as $k$-algebras. If we consider the coalgebra structures of $kG$ and $kH$ as in Example~\ref{exam:bialgebra}, then the twisting map is not compatible with the comultiplications, namely diagram \eqref{diag:tau-comultiplications} does not commute.

Alternatively, we can consider the Frobenius coalgebra structure on the group algebras where the counits $\epsilon_G:kG \to k$ and $\epsilon_H:kH \to k$ and comultiplications $\Delta_G: kG \to kG\otimes kG$ and $\Delta_H: kH \to kH\otimes kH$ are given by linearly extending $\epsilon_G(r) = \delta_{r,1}$ and $\epsilon_H(s) = \delta_{s,1}$ for $r\in G$ and $s\in H$, and
\begin{equation*}
\Delta_G(g) = \sum_{r\in G} gr\otimes r^{-1},\quad \Delta_H(h) = \sum_{s\in H} hs\otimes s^{-1}.
\end{equation*}
The above coalgebra structures coincides with the one in Example~\ref{exam:Frobenius} by Remark~\ref{rema:frob-alg-p-c-ring}. These comultiplications are compatible with the twisting map, namely diagram \eqref{diag:tau-comultiplications} commutes. Further, the Frobenius algebra structure obtained on $kG\otimes_{\tau} kH$ recovers exactly the Frobenius algebra structure on $k(H\rtimes_\varphi G)$ given by $\epsilon_{H\rtimes_\varphi G}:k(H\rtimes_\varphi G) \to k$ and $\Delta_{H\rtimes_\varphi G}: k(H\rtimes_\varphi G) \to k(H\rtimes_\varphi G)\otimes k(H\rtimes_\varphi G)$ as
\begin{equation*}
\epsilon_{H\rtimes_\varphi G}(s,r) = \delta_{(s,r),(1,1)},\quad \Delta_{H\rtimes_\varphi G}(h,g) = \sum_{(s,r)\in H\rtimes_\varphi G} (h,g)(s,r)\otimes (s,r)^{-1}.
\end{equation*}
\end{exam}

\begin{coro}\label{coro:ttp-cttp-separable-Frobenius}
Let $A$ and $B$ be separable Frobenius algebras over $k$, let $\Gamma_A:A\to A\otimes A$ and $\Gamma_B:B\to B\otimes B$ the right inverses of the respective multiplications. Let $\tau:B\otimes A\rightarrow A\otimes B$ be a twisting map such that diagram~\eqref{diag:tau-comultiplications} commutes. Then $(A\otimes_{\tau} B,\nabla_{A\otimes_{\tau} B},\eta_{A\otimes_{\tau} B},\Delta_{A\otimes_{\tau} B},\epsilon_{A\otimes_{\tau} B})$ is a separable Frobenius algebra if and only if diagrams~\eqref{diag:tau-commutes-sectionA} and \eqref{diag:tau-commutes-sectionB} commute.
\end{coro}
\begin{proof}
Apply Theorem~\ref{theo:ttp-cttp-always-Frobenius} and Theorem~\ref{theo:twist-separable}.
\end{proof}

\begin{coro}\label{coro:ttp-cttp-special-Frobenius}
Let $A$ and $B$ be special Frobenius algebras over $k$. Let $\tau:B\otimes A\rightarrow A\otimes B$ be a twisting map such that diagram~\eqref{diag:tau-comultiplications} commutes. Then, the Frobenius algebra $(A\otimes_{\tau} B,\nabla_{A\otimes_{\tau} B},\eta_{A\otimes_{\tau} B},\Delta_{A\otimes_{\tau} B},\epsilon_{A\otimes_{\tau} B})$ is special.
\end{coro}
\begin{proof}
By Theorem~\ref{theo:ttp-cttp-always-Frobenius}, $(A\otimes_{\tau} B,\nabla_{A\otimes_{\tau} B},\eta_{A\otimes_{\tau} B},\Delta_{A\otimes_{\tau} B},\epsilon_{A\otimes_{\tau} B})$ is indeed a Frobenius algebra. Since $A$ and $B$ are special Frobenius algebras, by Proposition~\ref{prop:frob-special-separable-comultiplication}, $A$ and $B$ are both separable with $\Delta_A$ and $\Delta_B$ the right inverses of $\nabla_A$ and $\nabla_B$ respectively. The algebra $(A\otimes_{\tau} B,\nabla_{A\otimes_{\tau} B},\eta_{A\otimes_{\tau} B})$ is then separable and, by Theorem~\ref{theo:twist-separable}, a right inverse of its multiplication is $(\Delta_A \otimes\Delta_B)(1\otimes\tau^{-1}\otimes 1) = \Delta_{A\otimes_{\tau} B}$. Hence, $(A\otimes_{\tau} B,\nabla_{A\otimes_{\tau} B},\eta_{A\otimes_{\tau} B},\Delta_{A\otimes_{\tau} B},\epsilon_{A\otimes_{\tau} B})$ is a special Frobenius algebra by Proposition~\ref{prop:frob-special-separable-comultiplication}.
\end{proof}

The previous results together with Theorem~\ref{theo:twist-separable} are a noncommutative generalization of the fact that the tensor product of {\'e}tale algebras is an {\'e}tale algebra. Moreover, we achieved a very computationally efficient criterion for self-injectivity.

\begin{coro}[Criterion for self-injectivity]\label{coro:self-injective-criterion}
Let $\Lambda$ be a unital associative $k$-algebra, let $A$ and $B$ be $k$-subalgebras of $\Lambda$ via the injective $k$-algebra morphisms $\iota_A : A\to \Lambda$ and $\iota_B : B\to \Lambda$ such that $\nabla_{\Lambda} (\iota_A\otimes\iota_B) : A\otimes B\to \Lambda$ is an isomorphism of $k$ vector spaces. If $A$ and $B$ are Frobenius algebras and $\tau = (\nabla_{\Lambda} (\iota_A\otimes\iota_B))^{-1} \nabla_{\Lambda} (\iota_B\otimes\iota_A)$ makes diagram~\eqref{diag:tau-comultiplications} commute, then $\Lambda$ is self-injective.
\end{coro}
\begin{proof}
By \cite[Proposition 2.7]{MR1352565} the hypothesis give $\Lambda \cong A\otimes_{\tau} B$ as $k$-algebras. Then $\Lambda$ is a Frobenius algebra by Theorem~\ref{theo:ttp-cttp-always-Frobenius}, so it is self-injective.
\end{proof}

This is an extremely useful criterion that has multiple applications, including to twisted tensor products of truncated polynomial rings and finite dimensional semisimple algebras. For example, it can be used to prove that some quantum complete intersections are self-injective (see Corollary~\ref{coro:quantum-complete-Frobenius} for a more general statement). We conclude the section by explicitly giving the pairing and co-pairing of a twisted tensor product of Frobenius algebras.

\begin{prop}\label{prop:ttp-pairing-copairing}
Let $A$ and $B$ be Frobenius algebras over $k$ with pairings $\beta_A$ and $\beta_B$ and co-pairings $\alpha_A$ and $\alpha_B$, let $\tau:B\otimes A\rightarrow A\otimes B$ be a twisting map such that diagram~\eqref{diag:tau-comultiplications} commutes. Then
\begin{center}
\begin{tikzcd}
\beta_{A\otimes_\tau B}: A\otimes B\otimes A\otimes B \arrow{r}{1\otimes\tau\otimes 1} & A\otimes A\otimes B\otimes B \arrow{r}{\beta_A \otimes \beta_B} & k\otimes k \arrow{r}{\cong} & k,
\end{tikzcd}
\end{center}
\begin{center}
\begin{tikzcd}
\alpha_{A\otimes_\tau B}: k \arrow{r}{\cong} & k\otimes k \arrow{r}{\alpha_A\otimes\alpha_B} & A\otimes A\otimes B\otimes B \arrow{r}{1\otimes \tau^{-1}\otimes 1} & A\otimes B\otimes A\otimes B,
\end{tikzcd}
\end{center}
are an associative non-degenerate pairing and co-pairing of $A\otimes_{\tau} B$.
\end{prop}
\begin{proof}
Since $A\otimes_{\tau} B$ is a Frobenius algebra by Theorem~\ref{theo:ttp-cttp-always-Frobenius}, it has an associative pairing $\beta_{A\otimes_\tau B}$ and co-pairing $\alpha_{A\otimes_\tau B}$ as in Remark~\ref{rema:frob-alg-p-c-ring}. The claimed expression for $\alpha_{A\otimes_\tau B}$ is given by the following commutative diagram.
\begin{center}
\begin{tikzcd}
kk \arrow{rrr}{\alpha_A\alpha_B}  \arrow{rd}{\cong} \arrow[swap]{ddd}{\eta_A\eta_B} & & & AABB \arrow{r}{1\tau^{-1}1} \arrow{ddd}{1111}  & ABAB \arrow{ddd}{1111} \\
& kkkk \arrow{r}{\alpha_A11\alpha_B} \arrow[swap]{d}{1\eta_A\eta_B 1} & AAkkBB  \arrow{ur}{\cong} \arrow{d}{1\eta_A\eta_B1} & & \\
& kABk \arrow{r}{\alpha_A11\alpha_B} & AAABBB \arrow{rd}{1\nabla_A\nabla_B1}& & \\
AB \arrow{rrr}{\Delta_A\Delta_B} \arrow{ur}{\cong} & & & AABB \arrow{r}{1\tau^{-1}1} & ABAB
\end{tikzcd}
\end{center}
Similarly for $\beta_{A\otimes_\tau B}$ we have the following commutative diagram.
\begin{center}
\begin{tikzcd}
ABAB \arrow{r}{1\tau1} \arrow[swap]{ddd}{1111} & AABB \arrow{rrr}{\nabla_A\nabla_B} \arrow{dr}{\cong} \arrow[swap]{ddd}{1111} & & & AB \arrow{ld}{\cong} \arrow{ddd}{\epsilon_A\epsilon_B}\\
& & kAABBk \arrow{r}{1\nabla_A\nabla_B1}\arrow[swap]{d}{\eta_A1111\eta_B} & kABk \arrow{d}{\eta_A11\eta_B} & \\
& & AAABBB \arrow{r}{1\nabla_A\nabla_B1} \arrow{ld}{\nabla_A11\nabla_B} & AABB \arrow{rd}{\beta_A\beta_B} & \\
ABAB \arrow{r}{1\tau1} & AABB \arrow{rrr}{\beta_A\beta_B} & & & kk
\end{tikzcd}
\end{center}
\end{proof}

Although the pairing and co-pairing of $A\otimes_{\tau} B$ are readily attainable, it is not easy to compute the corresponding Nakayama automorphism.  Some natural candidates to consider are $\Theta_{A} \otimes \Theta_{B} : A\otimes B\to A\otimes B$ and $(\tau)(\Theta_{B} \otimes\Theta_{A})(\tau^{-1}): A\otimes B\to A\otimes B$, but it is hard to see that either of the following diagrams commute.
\begin{center}
\begin{tikzcd}
ABAB \arrow{r}{\Theta_{A} \Theta_{B} 1 1} \arrow[swap]{d}{\sigma_{13}\sigma_{24}} & ABAB \arrow{d}{\beta_{A\otimes_{\tau} B}}\\
ABAB \arrow{r}{\beta_{A\otimes_{\tau} B}} & k
\end{tikzcd} \quad \begin{tikzcd}
ABAB \arrow{r}{\tau^{-1} 1 1} \arrow[swap]{d}{\sigma_{13}\sigma_{24}} & BAAB \arrow{r}{\Theta_{B} \Theta_{A} 1 1} & BAAB \arrow{r}{\tau 1 1} & ABAB \arrow{d}{\beta_{A\otimes_{\tau} B}}\\
ABAB \arrow{rrr}{\beta_{A\otimes_{\tau} B}} & & & k
\end{tikzcd}
\end{center}
At the end of Section~\ref{sec:novel-Frobenius} we prove that these candidates coincide in the case of group algebras and strongly graded twists, and that both these diagrams commute. It would be interesting to know whether either of the proposed automorphisms is the Nakayama automorphism $\Theta_{A\otimes_{\tau} B}$ corresponding to the co-pairing $\beta_{A\otimes_{\tau} B}$ in general.

\section{Known and novel Frobenius algebra structures}\label{sec:novel-Frobenius}

In this section, we show that twisting by a bicharacter is compatible with graded Frobenius algebras having comultiplications of degree zero, and we establish sufficient and necessary conditions for the compatibility when the comultiplication is graded. We use this to show that certain quantum complete intersections inherit a symmetric Frobenius algebra structure when seen as twisted tensor products, and we construct noncommutative symmetric Frobenius algebras.

\begin{lemm}\label{lemm:bicharacter-graded-coproduct-d}
Let $A$ and $B$ be $k$ vector spaces graded by abelian groups $F$ and $G$ respectively, let $\Delta_A:A\rightarrow A\otimes A$ and $\Delta_B:B\rightarrow B\otimes B$ be $k$-linear graded maps of degree $d_A$ and $d_B$ respectively, let $t:F\otimes_{\mathbb{Z}} G \rightarrow k^{\times}$ be a homomorphism of abelian groups and denote $t(f\otimes_{\mathbb{Z}} g) = t^{\langle f | g\rangle}$ for all $f\in F$ and $g\in G$, let $\tau:B\otimes A\rightarrow A\otimes B$ be given by linearly extending $\tau(b\otimes a) = t^{\langle |a| | |b|\rangle} a\otimes b$ for all homogeneous $a\in A$ and $b\in B$. Then
\begin{enumerate}
\item diagram \eqref{diag:tau-commutes-cmA} commutes if and only if $t^{\langle |a| | |b| \rangle} = t^{\langle |a| + d_A | |b| \rangle}$ for all homogeneous $a\in A$ and $b\in B$,
\item diagram \eqref{diag:tau-commutes-cmB} commutes if and only if $t^{\langle |a| | |b| \rangle} = t^{\langle |a| | |b| + d_B \rangle}$ for all homogeneous $a\in A$ and $b\in B$.
\end{enumerate}
\end{lemm}
\begin{proof}
We first consider diagram \eqref{diag:tau-commutes-cmB}. For homogeneous $a\in A$ and $b\in B$, a computation yields
\begin{align*}
(1\otimes \Delta_B)(\tau)(b\otimes a) &= t^{\langle |a| | |b| \rangle} a \otimes \sum_{(b)}{b_{(1)}\otimes b_{(2)}},\\
(\tau\otimes 1)(1\otimes \tau)(\Delta_B\otimes 1)(b\otimes a) &= a \otimes \sum_{(b)}{t^{\langle |a| | |b_{(1)}| \rangle} t^{\langle |a| | |b_{(2)}| \rangle} b_{(1)}\otimes b_{(2)}}.
\end{align*}
Since $t$ is a homomorphism of abelian groups and $\Delta_B$ is graded of degree $d_B$ the above are equal if and only if $t^{\langle |a| | |b| \rangle} = t^{\langle |a| | |b_{(1)}| + |b_{(2)}| \rangle} = t^{\langle |a| | |b| + d_B \rangle}$, as desired. It follows analogously for diagram \eqref{diag:tau-commutes-cmA}.
\end{proof}

\begin{theo}\label{theo:twisting-bicharacter-Frobenius}
Let $A$ and $B$ be Frobenius algebras over $k$ graded by abelian groups $F$ and $G$ respectively, with comultiplications $\Delta_A:A\rightarrow A\otimes A$ and $\Delta_B:B\rightarrow B\otimes B$ being graded maps of degree $d_A$ and $d_B$ respectively. Let $t:F\otimes_{\mathbb{Z}} G \rightarrow k^{\times}$ be a homomorphism of abelian groups. Then, $A\otimes^{t} B$ is a Frobenius algebra if and only if $t^{\langle |a| | |b| \rangle} = t^{\langle |a| + d_A | |b| \rangle}$ and $t^{\langle |a| | |b| \rangle} = t^{\langle |a| | |b| + d_B \rangle}$ for all homogeneous $a\in A$ and $b\in B$.
\end{theo}
\begin{proof}
Diagram \eqref{diag:tau-comultiplications} commutes because of Lemma~\ref{lemm:bicharacter-graded-coproduct-d} and Proposition~\ref{prop:AtB-counital-coassociative}. This is a sufficient and necessary condition for $A\otimes^{t} B$ to be a Frobenius algebra by Theorem~\ref{theo:ttp-cttp-always-Frobenius} and Remark~\ref{rema:if-and-only-if}.
\end{proof}

Note that if $A$ and $B$ are connected self-injective $k$-algebras graded by abelian groups $F$ and $G$ then they are finite dimensional and Frobenius by \cite{MR1451692}. Hence if their comultiplications are graded maps of degree $d_A$ and $d_B$ respectively, and $t:F\otimes_{\mathbb{Z}} G \rightarrow k^{\times}$ is a homomorphism of abelian groups, then $A\otimes^{t} B$ is a Frobenius algebra if and only if $t^{\langle |a| | |b| \rangle} = t^{\langle |a| + d_A | |b| \rangle}$ and $t^{\langle |a| | |b| \rangle} = t^{\langle |a| | |b| + d_B \rangle}$ for all homogeneous $a\in A$ and $b\in B$ by Theorem~\ref{theo:twisting-bicharacter-Frobenius}.

\begin{rema}\label{rema:twisting-bicharacter-generated-1}
When $F = \langle f \rangle_{\Z}$ or $G = \langle g \rangle_{\Z}$ are generated by a single element, say of degree $|f| = 1 = |g|$, the bicharacters $t$ such that $A\otimes^{t} B$ is Frobenius can be characterized as the ones satisfying $(t^{\langle 1 | 1 \rangle})^{d_A} = 1 = (t^{\langle 1 | 1 \rangle})^{d_B}$, namely the order of $t^{\langle 1 | 1 \rangle}$ divides the greatest common divisor of $d_A$ and $d_B$.
\end{rema}

In particular, when $d_A = 0_F$ and $d_B = 0_G$ we have the following.

\begin{coro}\label{coro:twisting-bicharacter-Frobenius-0}
Let $A$ and $B$ be Frobenius algebras over $k$ graded by abelian groups $F$ and $G$ respectively, with comultiplications $\Delta_A:A\rightarrow A\otimes A$ and $\Delta_B:B\rightarrow B\otimes B$ being graded maps of degree $0$. Let $t:F\otimes_{\mathbb{Z}} G \rightarrow k^{\times}$ be a homomorphism of abelian groups. Then, $A\otimes^{t} B$ is a Frobenius algebra.
\end{coro}

We now use these results on quantum complete intersections.

\begin{exam}[Quantum complete intersection] \label{exam:quantum-complete-intersection} \cite[Section 2]{MR1608565}
Let $m_1,m_2\in\N$, $m_1,m_2 \geq 2$, and consider the truncated polynomial rings in one variable $k[x_1]/(x_1^{m_1})$ and $k[x_2]/(x_2^{m_2})$ with the twisting map $\tau:k[x_2]/(x_2^{m_2})\otimes k[x_1]/(x_1^{m_1})\rightarrow k[x_1]/(x_1^{m_1})\otimes k[x_2]/(x_2^{m_2})$ given by linearly extending $\tau(x_2\otimes x_1) = q x_1\otimes x_2$ for some non-zero $q\in k$. The algebra
\begin{equation*}
k[x_1]/(x_2^{m_1})\otimes_{\tau}k[x_2]/(x_2^{m_2})\cong k_q[x,y]/(x_1^{m_1},x_2^{m_2}) = k\langle x_1,x_2\rangle/(x_1^{m_1}, x_2^{m_2}, q x_1 x_2 - x_2 x_1)
\end{equation*}
is called a quantum complete intersection. Note that $k[x_i]/(x_i^{m_i})$ for $i=1,2$ is a Frobenius algebra as in Example~\ref{exam:Frobenius}. More generally, given $n\in\N$, $\mathbf{m} = (m_1,\dots,m_n)\in\N^n$, $n,m_1,\dots,m_n \geq 2$, and $\mathbf{q} = (q_{ij}) \in M_{n}(k^{\times})$ as in Example~\ref{exam:kx-twist-not-bialg-quantum-plane}, the above can be extended to
\begin{equation*}
\Lambda_{\mathbf{q},\mathbf{m}}^{n} \coloneqq  k_{\mathbf{q}}[x_1,\dots,x_n]/(x_1^{m_1},\dots,x_n^{m_n}) = \frac{k\langle x_1,\dots, x_n\rangle}{\left(x_i^{m_i},x_i x_j - q_{ij} x_j x_i \right)_{i,j\in\{1,\dots,n\}}}
\end{equation*}
which account for all quantum complete intersections.
\end{exam}

Note that $\Lambda_{\mathbf{q},\mathbf{m}}^{n}$ can be obtained via twists by a bicharacter in two equivalent ways. To interpret $\Lambda_{\mathbf{q},\mathbf{m}}^{n}$ as an iterated twisted tensor product (see \cite[Section 2]{MR2458561}) of twists by a bicharacter we use the $\Z$-grading on $k[x_i]/(x_i^{m_i})$ given by setting $|x_i| = 1$ for all $i = 1,\dots,n$. Consider $t_{i}^{j}:\Z\otimes_{\Z} \Z \rightarrow k^{\times}$ given by $t_{i}^{j}(r\otimes_{\Z} s) = q_{ji}^{rs}$ for all $r,s\in\Z$ and all $i,j = 1,\dots,n$, which yield twisting maps $\tau_{ij} : k[x_i]/(x_i^{m_i})\otimes k[x_j]/(x_j^{m_j}) \to k[x_j]/(x_j^{m_j})\otimes k[x_i]/(x_i^{m_i})$ as in Example~\ref{exam:twisting-bicharacter}. Now
\begin{equation*}
\Lambda_{\mathbf{q},\mathbf{m}}^{n} \cong k[x_1]/(x_1^{m_1})\otimes^{t_{1}^{2}} k[x_2]/(x_2^{m_2}) \otimes^{t_{2}^{3}} \cdots \otimes^{t_{n-2}^{n-1}} k[x_{n-1}]/(x_{n-1}^{m_{n-1}}) \otimes^{t_{n-1}^{n}} k[x_n]/(x_n^{m_n})
\end{equation*}
denotes the $k$-algebra with the expected unit and multiplication. In particular
\begin{center}
\begin{tikzcd}
\nabla_{\Lambda_{\mathbf{q},\mathbf{m}}^{n}} \coloneqq (\nabla_1\otimes\cdots\otimes \nabla_n) \tau_{n,n-1} \tau_{n,n-1,n-2} \cdots \tau_{n,\dots,2} \tau_{n,\dots,1}
\end{tikzcd}
\end{center}
where $\nabla_i : k[x_i]/(x_i^{m_i})\otimes k[x_i]/(x_i^{m_i}) \to k[x_i]/(x_i^{m_i})$ is the usual multiplication and
\begin{center}
\begin{tikzcd}
\tau_{n,\dots,i} \coloneqq (1^{\otimes 2i-1} \otimes \tau_{i+1 i} \otimes 1^{\otimes 2(n-i)-1}) \cdots (1^{\otimes n+i-2} \otimes \tau_{n i} \otimes 1^{\otimes n-i})
\end{tikzcd}
\end{center}
for all $i = 1,\dots,n-1$. To interpret $\Lambda_{\mathbf{q},\mathbf{m}}^{n}$ as a single twisting by a bicharacter, observe
\begin{equation*}
\Lambda_{\mathbf{q},\mathbf{m}}^{n} \cong \Lambda_{(q_{ij})_{i,j\in\{1,\dots,n-1 \}},(m_1,\dots,m_{n-1})}^{n-1}\otimes^{t_{1,\dots,n-1}^n} k[x_n]/(x_n^{m_n})
\end{equation*}
where $t_{1,\dots,n-1}^n(x_1^{a_1}\cdots x_{n-1}^{a_{n-1}} \otimes x_n^{a_n}) \coloneqq \prod_{j=1}^{n-1}{q_{jn}^{a_j a_n}}$, see Bergh and Oppermann \cite[Lemma 5.1.]{MR2450729}. We can now partially recover the fact that quantum complete intersections are Frobenius \cite[Lemma 3.1]{MR2516162}.

\begin{coro} \label{coro:quantum-complete-Frobenius}
Let $n\in\N$, $\mathbf{m} = (m_1,\dots,m_n)\in\N^n$, $n,m_1,\dots,m_n \geq 2$, and $\mathbf{q} = (q_{ij}) \in M_{n}(k^{\times})$ such that $q_{ii} = 1$ and $q_{ij} q_{ji} = 1$ for all $1\leq i,j\leq n$. If $q_{ij}$ is a root of unity whose order divides $\gcd(m_i - 1, m_j - 1)$ for all $i,j = 1,\dots,n$, then the quantum complete intersections $\Lambda_{\mathbf{q},\mathbf{m}}^{n}$ are Frobenius algebras.
\end{coro}
\begin{proof}
Considering Remark~\ref{rema:twisting-bicharacter-generated-1} for the $\Z$-grading on $k[x_i]/(x_i^{m_i})$, that $t_{i}^{j}( 1 \otimes_{\Z} 1) = q_{ji}$, and that the comultiplication provided in Example~\ref{exam:quantum-complete-intersection} is graded of degree $m_i - 1$, then $k[x_i]/(x_i^{m_i})\otimes^{t_{i}^{j}} k[x_j]/(x_j^{m_j})$ is a Frobenius algebra for all $i,j = 1,\dots,n$ by Theorem~\ref{theo:twisting-bicharacter-Frobenius}. An induction argument finishes the proof.
\end{proof}

The co-pairing $\beta_{\Lambda_{\mathbf{q},\mathbf{m}}^{n}}$ can be computed by successively applying the twisting maps as
\begin{align*}
\beta_{\Lambda_{\mathbf{q},\mathbf{m}}^{n}}(x_1^{a_1} \otimes \cdots \otimes x_n^{a_n} \otimes x_1^{b_1} \otimes \cdots \otimes x_n^{b_n}) &= \prod_{j = 2}^{n} \prod_{i = 1}^{j-1}{q_{ji}^{a_{i} b_{j}}} \epsilon_{\Lambda_{\mathbf{q},\mathbf{m}}^{n}} (x_1^{a_1 + b_1} \otimes \cdots \otimes x_n^{a_n + b_n})\\
 &= \prod_{j = 2}^{n} \prod_{i = 1}^{j-1}{q_{ji}^{a_{i} b_{j}}} \prod_{l = 1}^{n}{\delta_{a_l + b_l,m_l - 1}}
\end{align*}
whence it will be non-zero exactly when $a_l + b_l = m_l - 1$ for all $l = 1,\dots,n$. In that case
\begin{equation*}
\beta_{\Lambda_{\mathbf{q},\mathbf{m}}^{n}}(x_1^{b_1} \otimes \cdots \otimes x_n^{b_n} \otimes x_1^{a_1} \otimes \cdots \otimes x_n^{a_n}) = \prod_{j = 2}^{n} \prod_{i = 1}^{j-1}{q_{ji}^{b_{i} a_{j}}}
\end{equation*}
and since $b_i a_j = (m_i - 1 - a_i)(m_j - 1 - b_j)$ then $q_{ji}^{b_i a_j} = q_{ji}^{a_i b_j}$. Thus
\begin{equation*}
\beta_{\Lambda_{\mathbf{q},\mathbf{m}}^{n}}(x_1^{a_1} \otimes \cdots \otimes x_n^{a_n} \otimes x_1^{b_1} \otimes \cdots \otimes x_n^{b_n}) = \beta_{\Lambda_{\mathbf{q},\mathbf{m}}^{n}}(x_1^{b_1} \otimes \cdots \otimes x_n^{b_n} \otimes x_1^{a_1} \otimes \cdots \otimes x_n^{a_n})
\end{equation*}
and $\Lambda_{\mathbf{q},\mathbf{m}}^{n}$ is symmetric. These Frobenius structures are exactly the ones given by \cite[Lemma 3.1]{MR2516162}, since when the order of $q_{ij}$ divides $\gcd(m_i - 1, m_j - 1)$ for all $i,j = 1,\dots,n$ then their Nakayama automorphism is the identity map, whence their Frobenius algebras are also symmetric.

\begin{coro} \label{coro:p-group-Frobenius}
Let $k$ be a field of characteristic $p > 0$. Let $n\in\N$, $\mathbf{m} = (p,\dots,p)$, and $\mathbf{q} = (q_{ij}) \in M_{n}(k^{\times})$ such that $q_{ii} = 1$ and $q_{ij} q_{ji} = 1$ for all $1\leq i,j\leq n$. Then the quantum complete intersections $\Lambda_{\mathbf{q},\mathbf{m}}^{n}$ are Frobenius algebras.
\end{coro}
\begin{proof}
Since $k$ has characteristic $p$ then $k[x_i]/(x_i^{p}) \cong kC_p$ where $C_p$ is the cyclic group of order $p$. Now $kC_p$ has a Frobenius algebra structure where the comultiplications are graded maps of degree $0$, whence $k[x_i]/(x_i^{m_i})\otimes^{t_{i}^{j}} k[x_j]/(x_j^{m_j})$ is a Frobenius algebra for all $i,j = 1,\dots,n$ by Corollary~\ref{coro:twisting-bicharacter-Frobenius-0}. An induction argument finishes the proof.
\end{proof}

All the quantum complete intersections encompassed by Corollaries~\ref{coro:quantum-complete-Frobenius} and~\ref{coro:p-group-Frobenius} satisfy the finite generation hypothesis by \cite[Theorem 5.5.]{MR2450729}. Moreover, if $q_{ij}$ are not roots of unity whose order divides $\gcd(m_i - 1, m_j - 1)$, then the twists giving $\Lambda_{\mathbf{q},\mathbf{m}}^{n}$ are not compatible with the coproducts of Example~\ref{exam:quantum-complete-intersection}. Thus our specific description of quantum complete intersections as twists by a bicharacter does not impose a Frobenius algebra structure on $\Lambda_{\mathbf{q},\mathbf{m}}^{n}$. It would be interesting to know if there are other interpretations of quantum complete intersections as twisted tensor products that give the Frobenius algebra structures found by Bergh~\cite[Lemma 3.1]{MR2516162}.

We conclude with a brief study of the Frobenius algebra structures on twisted tensor products of group algebras with strongly graded twists. Let $G$ and $H$ be finite groups, let $\tau : kH\otimes kG \to kG\otimes kH$ be a strongly graded twisting map. Consider the gradings $kG = \bigoplus_{r\in G}{k_r}$ and $kH = \bigoplus_{s\in H}{k_s}$, since $\tau : kH\otimes kG \to kG\otimes kH$ is strongly graded then $\tau(h\otimes g) = \lambda_{h,g} g\otimes h$ for some $\lambda_{h,g} \in k^{\times}$. Then
\begin{align*}
(1\otimes \Delta_H)(\tau)(h\otimes g) &= \lambda_{h,g} g \otimes \sum_{s\in H}{s h\otimes s^{-1}},\\
(\tau\otimes 1)(1\otimes \tau)(\Delta_H\otimes 1)(h\otimes g) &= g \otimes \sum_{s\in H}{\lambda_{s^{-1},g} \lambda_{sh,g} s h\otimes s^{-1}},\\
(\Delta_G\otimes 1)(\tau)(h\otimes g) &= \lambda_{h,g} \sum_{r\in G}{r g\otimes r^{-1}} \otimes h,\\
(1\otimes \tau)(\tau \otimes 1)(1\otimes \Delta_G)(h\otimes g) &= \sum_{r\in G}{\lambda_{h,rg} \lambda_{h,r^{-1}} r g\otimes r^{-1}} \otimes h.
\end{align*}
If $\lambda_{h,g}=\lambda_{s^{-1},g}\lambda_{sh,g}$ and $\lambda_{h,g}=\lambda_{h,r^{-1}}\lambda_{h,rg}$ for all $g,r\in G$ and $h,s\in H$ then the first two expressions above are equal, and the last two are also equal. 
Thus diagram \eqref{diag:tau-comultiplications} commutes and $kG\otimes_{\tau} kH$ is Frobenius by Theorem~\ref{theo:ttp-cttp-always-Frobenius}. This is slightly more general than Corollary~\ref{coro:twisting-bicharacter-Frobenius-0} applied to $kG$ and $kH$, since $\tau$ does not have to come from a bicharacter. Let $a,b\in G$ and $c,d \in H$, we can explicitly compute the co-pairing $\beta_{kG\otimes_{\tau} kH} = \Delta_{k}(\beta_{kG} \otimes \beta_{kH})(1\otimes \tau \otimes 1)$ given in Proposition~\ref{prop:ttp-pairing-copairing}.
\begin{align*}
\beta_{kG\otimes_{\tau} kH} (a\otimes c\otimes b\otimes d) &= \Delta_{k}(\beta_{G} \otimes \beta_{H})(1\otimes \tau \otimes 1)(a\otimes c\otimes b\otimes d)\\
 &= \Delta_{k}(\beta_{G} \otimes \beta_{H})(\lambda_{c,b} a\otimes b\otimes c\otimes d) = \lambda_{c,b} \delta_{ab,1} \delta_{cd,1}
\end{align*}
In particular it is non-zero exactly when $ab = 1_{G}$ and $cd = 1_{H}$, in which case $\lambda_{c,b} = \lambda_{d^{-1},a^{-1}} = \lambda_{d,a}$. Thus
\begin{align*}
\beta_{kG\otimes_{\tau} kH} \sigma_{24} \sigma_{13} (a\otimes c\otimes b\otimes d) = \lambda_{d,a} \delta_{ba,1} \delta_{dc,1} = \lambda_{c,b} \delta_{ab,1} \delta_{cd,1} = \beta_{kG\otimes_{\tau} kH} (a\otimes c\otimes b\otimes d)
\end{align*}
and $kG\otimes_{\tau} kH$ is a symmetric Frobenius algebra.

\begin{exam}[Symmetric twisted tensor product of Frobenius algebras]\label{exam:twist-symmetric-Frobenius}
Let $C_2$ be the cyclic group of order $2$ with multiplicative generator $g$. Now $kC_2$ is a symmetric Frobenius algebra with co-pairing $\beta : kC_2\otimes kC_2\to k$ given by
\begin{equation*}
\beta\left(\sum_{a,b\in C_2}{\lambda_{a,b} a\otimes b} \right) = \lambda_{1,1} + \lambda_{g,g}.
\end{equation*}
Let $A = B = kC_2$, let $\tau:B\otimes A\to A\otimes B$ be the twisting map induced by the bicharacter $t:C_2\otimes_{\mathbb{Z}} C_2 \to k^{\times}$ defined as $t(1\otimes_{\mathbb{Z}} 1) = t(g\otimes_{\mathbb{Z}} 1) = t(1\otimes_{\mathbb{Z}} g) = 1$ and $t(g\otimes_{\mathbb{Z}} g) = -1$. Now $A\otimes_{\tau} B$ is a symmetric Frobenius algebra with co-pairing $\beta_{kC_2\otimes_{\tau} kC_2}:kC_2\otimes_{\tau} kC_2 \otimes kC_2\otimes_{\tau} kC_2 \to k$ given by $\beta_{kC_2\otimes_{\tau} kC_2}(1\otimes 1\otimes 1\otimes 1) = \beta_{kC_2\otimes_{\tau} kC_2}(1\otimes g\otimes 1\otimes g) = \beta_{kC_2\otimes_{\tau} kC_2}(g\otimes 1\otimes g\otimes 1) = 1$, $\beta_{kC_2\otimes_{\tau} kC_2}(g\otimes g\otimes g\otimes g) = -1$, and $\beta_{kC_2\otimes_{\tau} kC_2}(a\otimes c\otimes b\otimes d) = 0$ for other choices of $a,b,c,d\in kC_2$.
\end{exam}

It is clear that commutative Frobenius algebras are always symmetric, but the converse is not true. The above provides a systematic construction of noncommutative symmetric Frobenius algebras.

\section*{Acknowledgements}

We thank Nicol{\'a}s Andruskiewitsch for giving the talk that inspired this work. We thank Gigel Militaru, Sarah Witherspoon, and Harshit Yadav for useful comments and remarks. We thank the organizers of the 20th International Conference on Representations of Algebras for enabling this collaboration. We thank Julia Pevtsova and Ralf Schiffler for travel support via their NSF grant DMS-2004170. We thank the anonymous referee for providing valuable comments, corrections, and suggestions that helped improve and clarify this manuscript. The first author would like to thank the Hausdorff Research Institute for Mathematics for its hospitality during the writing of this work, funded by the Deutsche Forschungsgemeinschaft (DFG, German Research Foundation) under Germany's Excellence Strategy - EXC-2047/1 - 390685813. The first author was also supported by an AMS-Simons travel grant. The second author would like to gratefully acknowledge that their work is supported by a grant from the Simons Foundation Targeted Grant (917524) to the Pacific Institute for the Mathematical Sciences.


\bibliographystyle{alpha}
\bibliography{ref}

\end{document}